\documentclass[a4paper,11pt]{article}
\usepackage[latin1]{inputenc}
\usepackage[hmargin={25mm,25mm},vmargin={25mm,30mm}]{geometry}
\usepackage{amsmath,amsfonts,amsthm,amssymb}
\usepackage{cancel}
\usepackage{comment}
\usepackage{enumerate}
\usepackage{graphicx}
\usepackage[ocgcolorlinks,allcolors={blue}]{hyperref}
\usepackage{xcolor}
\usepackage{authblk}
\usepackage{pgfplots}
\DeclareMathAlphabet{\mathpgoth}{OT1}{pgoth}{m}{n}
\usepackage{multirow}
\usepackage{booktabs}
\usepackage[compact,small,nobottomtitles]{titlesec}
\usepackage{diagbox}

\definecolor{labelkey}{rgb}{0.6,0,1}

\usepackage[font={footnotesize}]{subcaption}
\usepackage[absolute,showboxes,overlay]{textpos}
\textblockorigin{0pt}{0pt}   
\definecolor{darkbrown}{HTML}{996633}
\newcommand{\logLogSlopeTriangle}[5]
{
    \pgfplotsextra
    {
        \pgfkeysgetvalue{/pgfplots/xmin}{\xmin}
        \pgfkeysgetvalue{/pgfplots/xmax}{\xmax}
        \pgfkeysgetvalue{/pgfplots/ymin}{\ymin}
        \pgfkeysgetvalue{/pgfplots/ymax}{\ymax}

        \pgfmathsetmacro{\xArel}{#1}
        \pgfmathsetmacro{\yArel}{#3}
        \pgfmathsetmacro{\xBrel}{#1-#2}
        \pgfmathsetmacro{\yBrel}{\yArel}
        \pgfmathsetmacro{\xCrel}{\xArel}
        \pgfmathsetmacro{\lnxB}{\xmin*(1-(#1-#2))+\xmax*(#1-#2)} % in [xmin,xmax].
        \pgfmathsetmacro{\lnxA}{\xmin*(1-#1)+\xmax*#1} % in [xmin,xmax].
        \pgfmathsetmacro{\lnyA}{\ymin*(1-#3)+\ymax*#3} % in [ymin,ymax].
        \pgfmathsetmacro{\lnyC}{\lnyA+#4*(\lnxA-\lnxB)}
        \pgfmathsetmacro{\yCrel}{\lnyC-\ymin)/(\ymax-\ymin)} % THE IMPROVED EXPRESSION WITHOUT 'DIMENSION TOO LARGE' ERROR.

        \coordinate (A) at (rel axis cs:\xArel,\yArel);
        \coordinate (B) at (rel axis cs:\xBrel,\yBrel);
        \coordinate (C) at (rel axis cs:\xCrel,\yCrel);

        \draw[black]   (A)-- node[pos=0.5,anchor=north] {\scriptsize{1}}
                    (B)-- 
                    (C)-- node[pos=0.,anchor=west] {\scriptsize{\color{#5}#4}} %% node[pos=0.5,anchor=west] {#4}
                    (A);
    }
}

\usepackage{stackengine}

\usepackage{newtxtext,newtxmath}
\usepackage[style=numeric-comp,maxbibnames=99,firstinits=true]{biblatex}
\bibliography{hholli}
\usepackage{here}
\newtheorem{theorem}{Theorem}
\newtheorem{proposition}[theorem]{Proposition}
\newtheorem{lemma}[theorem]{Lemma}

\theoremstyle{remark}
\newtheorem{remark}[theorem]{Remark}
\theoremstyle{definition}
\newtheorem{assumption}{Assumption}

\newtheorem{example}[theorem]{Example}

\newcommand{\email}[1]{\href{mailto:#1}{#1}}

\usepackage[normalem]{ulem}
\newcounter{corr}
\definecolor{violet}{rgb}{0.580,0.,0.827}
\newcommand{\corr}[3]{\typeout{Warning : a correction remains in page
    \thepage}
  \stepcounter{corr}        
	      {\color{blue}\ifmmode\text{\,\sout{\ensuremath{#1}}\,}\else\sout{#1}\fi}
        {\color{red}#2}
        {\color{violet} #3}
}

\title{Improved error estimates for Hybrid High-Order discretizations of Leray--Lions problems}
\author[1]{Daniele A. Di Pietro \footnote{\email{daniele.di-pietro@umontpellier.fr}}}
\author[2]{J\'{e}r\^{o}me Droniou \footnote{\email{jerome.droniou@monash.edu}}}
\author[1]{Andr\'{e} Harnist \footnote{\email{andre.harnist@umontpellier.fr}, corresponding author}}	
\affil[1]{IMAG, Univ Montpellier, CNRS, Montpellier, France}
\affil[2]{School of Mathematics, Monash University, Melbourne, Australia}

\def\b{\boldsymbol}
\newcommand*\cst[1]{\mathrm{#1}}

\newcommand{\R}{\mathbb{R}} %reals
\newcommand{\N}{\mathbb{N}} %naturals
\newcommand{\Poly}{\mathbb{P}} %polynomials
\newcommand\dU[2]{{\u U}_{#1}^{#2}} %discrete velocity space
\renewcommand\u{\underline} %discrete scalar function
\newcommand{\T}{\mathcal{T}} %elements
\newcommand{\F}{\mathcal{F}} %faces
\newcommand{\Fb}{\F_h^\cst{b}} %boundary faces
\newcommand{\Fi}{\F_h^\cst{i}} %interfaces

\newcommand{\res}[1]{\ \!\!_{|_{#1}}} %restriction
\newcommand{\dgrad}[2]{\b{\cst{G}}^{#1}_{#2}} %discrete gradient
\newcommand{\drec}[2]{\cst{r}^{#1}_{#2}} %discrete reconstruction operator
\newcommand{\dfbres}[2]{\Delta^{#1}_{#2}} %discrete face-based residual operator
\newcommand{\GRAD}{\b\nabla} %gradient
\newcommand{\DIV}{\b\nabla{\cdot}} %divergence for matrices
\newcommand\I[2]{\u{I}_{#1}^{#2}} %interpolator
\newcommand\proj[2]{\pi_{#1}^{#2}} %scalar projection
\newcommand\PROJ[2]{\b{\pi}_{#1}^{#2}} %projection for vectors or matrices
\newcommand\stress{\b\sigma} %shear stress-strain rate function

\newcommand\sob{p}

\newcommand\stab{\cst{S}}

\DeclareMathOperator*{\essinf}{ess\,inf}

\begin{document}

\maketitle

\begin{abstract}
  We derive novel error estimates for Hybrid High-Order (HHO) discretizations of Leray--Lions problems set in $W^{1,p}$ with $p\in (1,2]$.
  Specifically, we prove that, depending on the degeneracy of the problem, the convergence rate may vary between $(k+1)(p-1)$ and $(k+1)$, with $k$ denoting the degree of the HHO approximation.
  These regime-dependent error estimates are illustrated by a complete panel of numerical experiments.
  \medskip\\
  \textbf{Keywords:} Hybrid High-Order methods, degenerate Leray--Lions problems, regime-dependent error estimates
  \smallskip\\
  \textbf{MSC2018 classification:} 65N08, 65N30, 65N12 
\end{abstract}

\section{Introduction}
\label{intro}

We consider Hybrid High-Order (HHO) approximations of Leray--Lions problems set in $W^{1,\sob}$ with $\sob\in(1,2]$.
For this class of problems, negative powers of the gradient of the solution can appear in the flux.
Depending on the expression of the latter, this can lead to a (local) degeneracy of the problem when the gradient of the solution vanishes or becomes large.

In this work, we prove novel error estimates that highlight the dependence of the convergence rate on the two possible cases of degeneracy simultaneously, and we do not differentiate these two degeneracies any longer.
Specifically, we show that, for the globally non-degenerate case, the energy-norm of the error converges as $h^{k+1}$, with $h$ denoting the mesh size and $k$ the degree of the HHO approximation.
In the globally degenerate case, on the other hand, the energy-norm of the error converges as $h^{(k+1)(\sob-1)}$, coherently with the estimate originally proved in \cite[Theorem 3.2]{Di-Pietro.Droniou:17*1}.
We additionally introduce, for each mesh element $T$ of diameter $h_T$, a dimensionless number $\eta_T$ that captures the local degeneracy of the model in $T$ and identifies the contribution of the element to the global error: from the fully degenerate regime, corresponding to a contribution in $\mathcal{O}(h_T^{(k+1)(p-1)})$, to the non-degenerate regime, corresponding to a contribution in $\mathcal{O}(h_T^{k+1})$, through all intermediate regimes.
Estimates depending on local regimes have been established in linear settings (see, e.g., \cite{Di-Pietro.Droniou.ea:18,Di-Pietro.Ern.ea:08} for advection--diffusion--reaction models and \cite{Botti.Di-Pietro.ea:18} for the Brinkman problem).
Such work has also been done for the Leray--Lions problems with Adaptative Finite Element methods in \cite{belenki:12}; to the best of our knowledge, their extension to HHO methods is entirely new.

Error estimates for the lowest-order conforming finite element approximation of the pure $p$-Laplacian have been known for quite some time; see, e.g., the founding work \cite{Glowinski.Marrocco:75}, in which $\mathcal O(h^{1/(3-p)})$ error estimates are obtained in the case $p\le 2$ considered here. These estimates were later improved in \cite{Barrett.Liu:93} to $\mathcal O(h)$, for solutions of high (and global) regularity -- in the space $W^{3,1}(\Omega)\cap C^{2,(2-p)/p}(\overline{\Omega})$. The above results have been extended to non-conforming finite elements in \cite{Liu.Yan:01}. A glaciology model is considered in \cite{Glowinski.Rappaz:03}, corresponding to a non-degenerate $p$-Laplace equation (with flux satisfying \eqref{eq:ass:sigma} below with $\delta=1$), and error estimates for the conforming finite element approximation have been obtained: $\mathcal O(h)$ if the solution is in $H^2(\Omega)$, and $O(h^{p/2})$ if it belongs to $W^{2,p}(\Omega)$. A common feature of all these studies, in which sharp error estimates are derived (which do not degrade too much as $p$ gets far from 2), is that they only consider low-order schemes on 2D triangular meshes and with continuity properties -- either all along the edges for the conforming method, or at the edges midpoints for the non-conforming method. 
To our knowledge, for higher-order methods that may involve fully discontinuous functions and are applicable to generic polytopal meshes, such as HHO, no sharp error estimates are known and only convergence in $h^{(k+1)(p-1)}$ has been established so far. This paper therefore bridges a gap between the results available for the low-order finite element methods and HHO methods.
Notice that, very recently, $\mathcal O(h^{(k+1)/(3-p)})$ error estimates have been obtained in \cite{carstensen:20} for an HHO method on standard simplicial meshes based on a stable gradient inspired by \cite{Di-Pietro.Droniou.ea:18}.
In passing, even though we focus on the HHO method, our approach could in all likelihood be extended to other polytopal methods such as, e.g., the Mimetic Finite Difference method \cite{Antonietti.Bigoni.ea:14} or the Virtual Element method \cite{Beirao-da-Veiga.Brezzi.ea:13}; see the preface of \cite{Di-Pietro.Droniou:20} for an up-to-date literature review on this subject.

The rest of the paper is organized as follows. In Section \ref{sec:continuous.setting} we establish the continuous setting, including novel assumptions on the flux function weaker than the ones considered in \cite[Section 3.1]{Di-Pietro.Droniou:17*1}.
In Section \ref{sec:discrete.setting} we briefly recall the discrete setting upon which rests the HHO scheme described in Section \ref{sec:scheme}.
The main result of this paper is contained Section \ref{sec:error.estimate}.
Finally, Section \ref{sec:num.res} contains a complete panel of numerical tests illustrating the effect of local degeneracy on the convergence rate.

\section{Continuous setting}\label{sec:continuous.setting}

\subsection{Flux function}

Let $\Omega \subset \R^d$, $d\in\N^*$, denote a bounded, connected, polytopal open set with Lipschitz boundary $\partial\Omega$. We consider the Leray--Lions problem, which consist in finding $u : \Omega \to \R$ such that
\begin{subequations}\label{eq:lr.continuous}
  \begin{alignat}{2} 
    -\DIV\stress(\cdot,\GRAD u)  &= f &\qquad& \mbox{  in  } \Omega,  \label{eq:lr.continuous:mass} \\
    u &=  0 &\qquad& \mbox{ on } \partial \Omega,
  \end{alignat}
\end{subequations} 
where $f : \Omega \to \R$ represents a volumetric force term, while $\stress : \Omega \times \R^d \to \R^d$ is the \emph{flux function}.
The flux function is possibly variable in space and depends on the \emph{potential} $u:\Omega\to\R$ only through its gradient.
The following assumptions characterize $\stress$.

\begin{assumption}[Flux function]\label{ass:stress}
  Let a real number $\sob \in (1,2]$ be fixed and denote by
    \[
    \sob' \coloneqq \frac{\sob}{\sob-1} \in \lbrack 2,+\infty)
    \]
    the conjugate exponent of $\sob$.
    The flux function satisfies
    \begin{subequations}\label{eq:ass:sigma}
      \begin{gather}
        \stress(\b x,\b 0) = \b 0 \text{ for almost every } \b x \in \Omega,\label{eq:ass-stress:0}
        \\
        \stress(\cdot,\b \xi):\Omega\to\R^d\mbox{ is measurable for all $\b \xi\in\R^d$}. 
      \end{gather}
      Moreover, there exist a \emph{degeneracy function} $\delta \in L^p(\Omega,[0,+\infty))$ and two real numbers $\sigma_\cst{hc},\sigma_\cst{sm} \in (0,+\infty)$ such that, for all $\b\tau,\b\eta \in \R^d$ and almost every $\b x \in \Omega$, we have the \emph{continuity} property
        \label{eq:power-framed:s.holder.continuity.strong.monotonicity}
        \begin{align} 
          \left|
          \stress(\b x,\b\tau)-\stress(\b x,\b\eta)
          \right| &\le \sigma_\cst{hc} \left(\delta(\b x)^\sob+|\b\tau|^\sob+|\b\eta|^\sob\right)^\frac{\sob-2}{\sob}| \b\tau-\b\eta |,\label{eq:power-framed:s.holder.continuity} 
        \end{align}
        and the \emph{strong monotonicity} property
        \begin{align}
          \left(\stress(\b x,\b\tau)-\stress(\b x,\b\eta)\right)\cdot(\b\tau-\b\eta)  \ge \sigma_\cst{sm}\left(\delta(\b x)^\sob+|\b\tau|^\sob+|\b\eta|^\sob\right)^\frac{\sob-2}{\sob}|\b\tau-\b\eta|^{2}.\label{eq:power-framed:s.strong.monotonicity}
        \end{align}    
    \end{subequations}
\end{assumption}

Some remarks are in order.

\begin{remark}[Flux at rest]
  Assumption \eqref{eq:ass-stress:0} expresses the fact that the flux at rest is zero, and can be relaxed taking $\stress(\cdot,\b 0)  \in L^{\sob'}(\Omega)^d$. This modification requires only minor changes in the analysis, not detailed for the sake of conciseness.
\end{remark}

\begin{remark}[Relations between the continuity and monotonicity constants]
  Inequalities \eqref{eq:power-framed:s.holder.continuity} and \eqref{eq:power-framed:s.strong.monotonicity} give
  \begin{equation}\label{eq:power-framed:constants.bound}
    \sigma_\cst{sm} \leq \sigma_\cst{hc}.
  \end{equation}
  Indeed, let $\b\tau \in \R^d$ be such that $|\b\tau| > 0$. Using the strong monotonicity \eqref{eq:power-framed:s.strong.monotonicity} (with $\b \eta = \b 0$) along with \eqref{eq:ass-stress:0}, the Cauchy--Schwarz inequality, and the continuity \eqref{eq:power-framed:s.holder.continuity} (again with $\b \eta = \b 0$) and \eqref{eq:ass-stress:0}, we infer that
  \[
  \begin{aligned}
    \sigma_\cst{sm}\left(\delta^\sob+|\b\tau|^\sob\right)^\frac{\sob-2}{\sob}|\b\tau|^{2} &\leq  \stress(\cdot,\b\tau)\cdot\b\tau\leq |\stress(\cdot,\b\tau)||\b\tau|\leq \sigma_\cst{hc}\left(\delta^\sob+|\b\tau|^\sob\right)^\frac{\sob-2}{\sob}| \b\tau|^2
  \end{aligned}
  \]
  almost everywhere in $\Omega$, hence \eqref{eq:power-framed:constants.bound}.
\end{remark}

\begin{remark}[Degenerate case]
 We note the following inequality:
  For all $x,y \in \R^n$, $n \in \N^*$, and all $\alpha \in [0,+\infty)$,
    \begin{equation}\label{eq:prolongement}
      (\alpha+|x|+|y|)^{p-2}|x-y| \le |x-y|^{p-1}.
    \end{equation}
    To prove \eqref{eq:prolongement}, notice that, if $\alpha+|x|+|y| > 0$, using a triangle inequality to write $|x|+|y|\ge |x-y|$ together with the fact that $\R \ni t \mapsto t^{\sob-2} \in \R$ is non-increasing (since $\sob<2$) and $\alpha\ge 0$, we infer that $\left(\alpha+|x|+|y|\right)^{\sob-2}\le |x-y|^{\sob-2}$, which, multiplying by $|x-y|$, gives \eqref{eq:prolongement}.
    Since \eqref{eq:prolongement} is valid when $\alpha+|x|+|y|>0$, we can extend the left-hand side by continuity (with value $0$) in the singular case $\alpha+|x|+|y|=0$ and this estimate remains valid.

    Inequality \eqref{eq:sum-power} below together with \eqref{eq:prolongement} ensures that properties \eqref{eq:power-framed:s.holder.continuity}--\eqref{eq:power-framed:s.strong.monotonicity} are well-formulated by extension also when $\delta(\b x)^p+|\b\tau|^p+|\b\eta|^p$ vanishes. 
    As a consequence, $\stress(\b x,\cdot):\R^d\to\R^d$ is continuous for a.e. $\b x\in\Omega$.
    The relation \eqref{eq:prolongement} will also play a key role in the proof of Theorem \ref{thm:error.estimate} below.
\end{remark}

\begin{remark}[Non-degenerate case]
	In \cite{Di-Pietro.Droniou:17*1}, an error estimate is given for broader versions of inequalities \eqref{eq:power-framed:s.holder.continuity}--\eqref{eq:power-framed:s.strong.monotonicity}. The novelty here lies in the introduction of the degeneracy function $\delta$, since it directly affects the convergence rate of the method. The crux of its intervention is located in the proof of Theorem \ref{thm:error.estimate}, and more precisely at \eqref{eq:consistency:ah:T2:1} where it prevents singularities. See Remark \ref{rem:ocv} for more details, see also Figure \ref{tab:num.res:flux} for a set of numerical results illustrating this phenomenon. 
\end{remark}

\begin{example}[$p$-Laplace flux function]
  A typical example of flux function is $\stress(\b x,\b\tau)=|\b\tau|^{p-2}\b\tau$, for which \eqref{eq:lr.continuous} is
the $p$-Laplace equation $-\DIV(|\GRAD u|^{p-2}\GRAD u)=f$. This flux function satisfies Assumption \ref{ass:stress} with
degeneracy function $\delta=0$, see e.g. \cite[Lemma 6.26]{Di-Pietro.Droniou:20}.
\end{example}

\begin{example}[Carreau--Yasuda flux function]\label{ex:Carreau--Yasuda}
  Another example of function $\stress$ which satisfies Assumption \ref{ass:stress}, inspired by the rheology of Carreau--Yosida fluids, is obtained setting, for almost every $\b x\in\Omega$ and all $\b\tau \in \R^d$,
  \begin{equation}\label{eq:Carreau--Yasuda}
    \stress(\b x,\b\tau) = \mu(\b x)\left(\delta(\b x)^{a(\b x)}+|\b\tau|^{a(\b x)}\right)^\frac{\sob-2}{a(\b x)}\b\tau,
  \end{equation}
  where $\mu : \Omega \to [\mu_-,\mu_+]$ is a measurable function with $\mu_-,\mu_+ \in (0,+\infty)$ corresponding to the \emph{local flow consistency index}, $\delta \in L^p(\Omega,[0,+\infty))$ is the \emph{degeneracy parameter}, $a : \Omega \to [a_-,a_+]$ is a measurable function with $a_-,a_+ \in (0,+\infty)$ expressing the \emph{local transition flow behavior index}, and $\sob \in (1,2]$ is the \emph{flow behavior index}.
  It was proved in \cite[Appendix A]{Botti.Castanon-Quiroz.ea:20} that $\stress$ is a $\sob$-power-framed function (with a straightforward analogy to replace the degeneracy constant $\sigma_\cst{de}$ therein by the degeneracy function $\delta$) with
    \[
    \sigma_\cst{hc} = 
      \frac{\mu_+}{\sob-1}2^{\left[-\left(\frac{1}{a_+}-\frac{1}{\sob}\right)^\ominus-1\right](\sob-2)+\frac{1}{\sob}} 
    \quad\text{and}\quad
    \sigma_\cst{sm} = 
      \mu_-(\sob-1)2^{\left(\frac{1}{a_-}-\frac{1}{\sob}\right)^\oplus(\sob-2)},
    \]
where $\xi^\oplus\coloneq\max(0;\xi)$ and $\xi^\ominus\coloneq-\min(0;\xi)$ denote, respectively, the positive and negative parts of a real number $\xi$. As a consequence, the flux function \eqref{eq:Carreau--Yasuda} matches Assumption \ref{ass:stress}.
\end{example}

\subsection{Weak formulation}\label{sec:weak.formulation}

We define the following space for the potential embedding the homogeneous boundary condition: 
\[
U \coloneqq W_0^{1,\sob}(\Omega).
\]
Assuming $f \in L^{\sob'}(\Omega)$, the weak formulation of problem \eqref{eq:lr.continuous} reads:
Find $u \in U$ such that
\begin{equation}\label{eq:lr.weak}
     a(u,v) = \displaystyle\int_\Omega f v \qquad \forall v \in U,
\end{equation}
where the \emph{diffusion function} $a : U \times U \to \R$ is defined such that, for all $v,w \in U$,
\begin{equation}\label{eq:a.b}
  a(w,v) \coloneqq \displaystyle\int_\Omega \stress(\cdot,\GRAD w) \cdot \GRAD v.
\end{equation}

\begin{proposition}[Well-posedness and a priori estimate]\label{prop:a-priori}
  Under Assumption \ref{ass:stress}, the continuous problem \eqref{eq:lr.weak} admits a unique solution $u \in U$ that satisfies the following a priori bound:
  \begin{equation}\label{eq:continuous.solution:bounds:uh}
    \begin{aligned}
    \|\GRAD u\|_{L^{\sob}(\Omega)^d} &\le
    \left(2^{\frac{2-\sob}{\sob}}C_{\rm P}\sigma_\cst{sm}^{-1}\| f \|_{L^{\sob'}(\Omega)}\right)^\frac{1}{\sob-1}
    + \min\!\left(\|\delta\|_{L^\sob(\Omega)};2^{\frac{2-\sob}{\sob}}C_{\rm P}\sigma_\cst{sm}^{-1}\|\delta\|_{L^\sob(\Omega)}^{2-\sob}\| f \|_{L^{\sob'}(\Omega)}\right).
      \end{aligned}
  \end{equation}
  where the real number $C_{\rm P}>0$, only depending on $\Omega$ and on $p$, is such that, for all $v\in W_0^{1,p}(\Omega)$, the Poincar\'e inequality $\|v\|_{L^p(\Omega)}\le C_{\rm P}\|\GRAD v\|_{L^p(\Omega)^d}$ holds.
\end{proposition}

\begin{proof}
  For the existence and uniqueness of a solution to \eqref{eq:lr.weak} see, e.g., \cite[Section 2.4]{Hirn:13}.
  To prove the a priori bound \eqref{eq:continuous.solution:bounds:uh}, use the strong monotonicity \eqref{eq:power-framed:s.strong.monotonicity} of $\stress$, \eqref{eq:lr.weak} written for $v=u$, and invoke the H\"older and Poincar\'e inequalities to write
  \[
  \begin{aligned}
    \sigma_\cst{sm}\left(
    \|\delta\|_{L^\sob(\Omega)}^\sob + \|\GRAD u\|_{L^\sob(\Omega)^d}^\sob
    \right)^\frac{\sob-2}{\sob} \|\GRAD u\|_{L^\sob(\Omega)^d}^2
    &\le  a(u,u)
    = \displaystyle\int_\Omega f u\le
    C_{\rm P} \| f \|_{L^{\sob'}(\Omega)}\|\GRAD u\|_{L^\sob(\Omega)^d},
  \end{aligned}
  \]
  that is,
  \begin{equation}\label{eq:continuous.solution:bounds:uh:1}
    \mathcal{N}\coloneqq \left(
    \|\delta\|_{L^\sob(\Omega)}^\sob + \|\GRAD u\|_{L^\sob(\Omega)^d}^\sob
    \right)^\frac{\sob-2}{\sob} \|\GRAD u\|_{L^\sob(\Omega)^d}
    \le C_{\rm P}  \sigma_\cst{sm}^{-1}\| f \|_{L^{\sob'}(\Omega)}.
  \end{equation}
  Observing that $\|\GRAD u\|_{L^\sob(\Omega)^d}\le 2^{\frac{2-\sob}{\sob}} \max\big(\|\GRAD u\|_{L^\sob(\Omega)^d};\|\delta\|_{L^\sob(\Omega)}\big)^{2-\sob}\mathcal{N} $, we obtain, enumerating the cases for the maximum and summing the corresponding bounds,
  \begin{equation}\label{eq:continuous.solution:bounds:u:2}
  \|\GRAD u\|_{L^\sob(\Omega)^d} \le (2^{\frac{2-\sob}{\sob}}\mathcal N)^{\frac{1}{\sob-1}} +2^{\frac{2-\sob}{\sob}}\|\delta\|_{L^\sob(\Omega)}^{2-\sob}\mathcal{N}.
  \end{equation}
  On the other hand, we have $\mathcal N \ge 2^\frac{p-2}{p}\|\GRAD u\|_{L^\sob(\Omega)^d}^{p-1}$ if $\|\GRAD u\|_{L^\sob(\Omega)^d} \ge \|\delta\|_{L^\sob(\Omega)}$. Thus we have, for any value of $\|\GRAD u\|_{L^\sob(\Omega)^d}$,
  \begin{equation}\label{eq:continuous.solution:bounds:u:3}
  \|\GRAD u\|_{L^\sob(\Omega)^d} \le (2^{\frac{2-\sob}{\sob}}\mathcal N)^{\frac{1}{\sob-1}} +\|\delta\|_{L^\sob(\Omega)}.
  \end{equation}
  Combining \eqref{eq:continuous.solution:bounds:uh:1} with the minimum of inequalities \eqref{eq:continuous.solution:bounds:u:2} and \eqref{eq:continuous.solution:bounds:u:3} gives \eqref{eq:continuous.solution:bounds:uh}.
\end{proof}

\section{Discrete setting}\label{sec:discrete.setting}

\subsection{Mesh}

For any set $X\subset\R^d$, denote by $h_X$ its diameter.
A polytopal mesh is defined as a couple $\mathcal M_h\coloneq(\T_h,\F_h)$, where $\T_h$ is a finite collection of polytopal elements $T\in\T_h$ such that $h=\max_{T\in\T_h}h_T$, while $\F_h$ is a finite collection of hyperplanar faces.
It is assumed henceforth that the mesh $\mathcal M_h$ matches the geometrical requirements detailed in \cite[Definition 1.7]{Di-Pietro.Droniou:20}.
Boundary faces lying on $\partial\Omega$ and internal faces contained in
$\Omega$ are collected in the sets $\F_h^{\rm b}$ and $\F_h^{\rm i}$, respectively.
For every mesh element $T\in\T_h$, we denote by $\F_T$ the subset of $\F_h$ collecting the faces that lie on the boundary $\partial T$ of $T$. For every face $F \in \F_h$, we denote by $\T_F$ the subset of $\T_h$ containing the one (if $F\in\F_h^{\rm b}$) or two (if $F\in\F_h^{\rm i}$) elements on whose boundary $F$ lies.
For each mesh element $T\in\T_h$ and face $F\in\F_T$, $\b n_{TF}$ denotes the (constant) unit vector normal to $F$ pointing out of $T$.

Our focus is on the $h$-convergence analysis, so we consider a sequence of refined meshes that is regular in the sense of \cite[Definition 1.9]{Di-Pietro.Droniou:20}, with regularity parameter uniformly bounded away from zero.
The mesh regularity assumption implies, in particular, that the diameter of a mesh element and those of its faces are comparable uniformly in $h$, and that the number of faces of one element is bounded above by an integer independent of $h$.

\subsection{Notation for inequalities up to a multiplicative constant}

To avoid the proliferation of generic constants, we write henceforth $a\lesssim b$ (resp., $a\gtrsim b$) for the inequality $a\le Cb$ (resp., $a\ge Cb$) with real number $C>0$ independent of $h$, of the parameters $\delta,\sigma_\cst{hc},\sigma_\cst{sm}$ in Assumption \ref{ass:stress}, and, for local inequalities, of the mesh element or face on which the inequality holds.
We also write $a\simeq b$ to mean $a\lesssim b$ and $b\lesssim a$.
The dependencies of the hidden constants are further specified when needed.

\subsection{Projectors and broken spaces}
  
Given $X \in \T_h \cup \F_h$ and $l \in \N$, we denote by $\Poly^l(X)$ the space spanned by the restriction to $X$ of scalar-valued, $d$-variate polynomials of total degree $\le l$.
The local $L^2$-orthogonal projector $\proj{X}{l} : L^{1}(X) \to \Poly^l(X)$ is defined such that, for all $v \in L^{1}(X)$,
\begin{equation}\label{eq:proj}
  \displaystyle\int_X (\proj{X}{l} v-v) w = 0 \qquad \forall w \in  \Poly^{l}(X).
\end{equation}
When applied to vector-valued functions in $L^1(X)^d$, the $L^2$-orthogonal projector mapping on $\Poly^l(X)^d$ acts component-wise and is denoted in boldface font as $\PROJ{X}{l}$.
Let $T\in\T_h$, $n\in[0,l+1]$, and $m\in[0,n]$.
The following $(n,\sob,m)$-approximation properties of $\proj{T}{l}$ hold:
For any $v\in W^{n,\sob}(T)$,
\begin{subequations}\label{eq:proj:app}
\begin{equation}\label{eq:proj:app:T}
  |v-\proj{T}{l}v|_{W^{m,\sob}(T)} \lesssim h_T^{n-m}|v|_{W^{n,\sob}(T)}.
\end{equation}
The above property will also be used in what follows with $\sob$ replaced by its conjugate exponent $\sob'$.
If, additionally, $n\ge 1$, we have the following $(n,\sob')$-trace approximation property:
\begin{equation}\label{eq:proj:app:F}
    \|v-\proj{T}{l}v\|_{L^{\sob'}(\partial T)}\lesssim h_T^{n-\frac{1}{\sob'}}|v|_{W^{n,\sob'}(T)}.
\end{equation}
\end{subequations}
The hidden constants in \eqref{eq:proj:app} are independent of $h$ and $T$, but possibly depend on $d$, the mesh regularity parameter, $l$, $n$, and $\sob$.
The approximation properties \eqref{eq:proj:app} are proved for integer $n$ and $m$ in \cite[Appendix A.2]{Di-Pietro.Droniou:17} (see also \cite[Theorem 1.45]{Di-Pietro.Droniou:20}), and can be extended to non-integer values using standard interpolation techniques (see, e.g., \cite[Theorem 5.1]{Lions.Magenes:72}).

The additional regularity on the exact solution in the error estimates will be expressed in terms of the broken Sobolev spaces
\[
W^{n,\sob}(\T_h)\coloneq\left\{ v\in L^\sob(\Omega)\ : \ v\res{T}\in W^{n,\sob}(T)\quad\forall T\in\T_h\right\}.
\]
The corresponding seminorm is such that $|v|_{W^{n,\sob}(\T_h)} \coloneq \big(\sum_{T\in\T_h}|v|_{W^{n,\sob}(T)}^\sob\big)^\frac{1}{\sob}$ for all $v \in W^{n,\sob}(\T_h)$.

%------------------------------------------------------------------------------%

\section{HHO discretization}\label{sec:scheme}

\subsection{Hybrid space and norms}

Let an integer $k\ge 0$ be fixed. The HHO space is, with usual notation,
\[
\dU{h}{k} \coloneqq \left\{
\u v_h = ((v_T)_{T \in \T_h},(v_F)_{F\in \F_h}) \ : \ v_T \in \Poly^k(T)\ \ \forall T \in \T_h\ \mbox{ and }\ v_F \in \Poly^k(F)\ \ \forall F \in \F_h \right\}.
\]
The interpolation operator $\I{h}{k} : W^{1,1}(\Omega) \to  \dU{h}{k} $ maps a function $v \in W^{1,1}(\Omega)$ on the vector of discrete unknowns $\I{h}{k}v$ defined as follows:
\[
  \I{h}{k} v \coloneqq ((\proj{T}{k} v\res{T})_{T \in \T_h},(\proj{F}{k} v\res{F})_{F \in \F_h}).
  \]
For all $T \in \T_h$, we denote by $\dU{T}{k}$ and $\I{T}{k}$ the restrictions of $\dU{h}{k}$ and $\I{h}{k}$ to $T$, respectively, and, for all $\u v_h \in \dU{h}{k}$, we let $\u v_T \coloneqq (v_T,(v_F)_{F\in \F_T}) \in \dU{T}{k}$ denote the vector collecting the discrete unknowns attached to $T$ and its faces.
Furthermore, for all $\u v_h \in \dU{h}{k}$, we define the broken polynomial field $v_h\in\Poly^k(\T_h)$ obtained patching element unknowns, that is,
\[
(v_h)\res{T} \coloneqq v_T\qquad\forall T \in \T_h.
\]

For all $q \in (1,+\infty)$, we define on $\dU{h}{k}$ the $W^{1,q}(\Omega)$-like seminorm $\| {\cdot} \|_{1,q,h}$ such that, for all $\u v_h \in \dU{h}{k}$,
\begin{subequations}\label{eq:norm.epsilon.r}
  \begin{gather}\label{eq:norm.epsilon.r.h}
    \| \u v_h \|_{1,q,h} \coloneqq \left(\displaystyle\sum_{T \in \T_h}\| \u v_T \|_{1,q,T}^q\right)^\frac{1}{q}
    \\\label{eq:norm.epsilon.r.T}
    \text{with 
      $\| \u v_T \|_{1,q,T} \coloneqq \left(\| \GRAD v_T \|^q_{L^q(T)^d} + \displaystyle\sum_{F \in \F_T} h_F^{1-q} \| v_F - v_T\|^q_{L^q(F)}\right)^\frac{1}{q}$
      for all $T \in \T_h$.
    }
  \end{gather}
\end{subequations}
The following boundedness property for $\I{T}{k}$ is proved in \cite[Proposition 6.24]{Di-Pietro.Droniou:20}:
For all $T \in \T_h$ and all $v \in W^{1,\sob}(T)$,
\begin{equation}\label{eq:I:boundedness}
  \|\I{T}{k} v\|_{1,\sob,T} \lesssim | v |_{W^{1,\sob}(T)},
\end{equation}
where the hidden constant depends only on $d$, the mesh regularity parameter, $\sob$, and $k$.

The discrete potential is sought in the subspace of $\dU{h}{k}$ embedding the homogeneous boundary condition:
\[
\dU{h,0}{k} \coloneqq \left\{ \u v_h = ((v_T)_{T \in \T_h},(v_F)_{F\in \F_h}) \in \dU{h}{k} \ : \ v_F = 0 \quad \forall F \in \Fb \right\}.
\]
The following discrete Poincar\'e inequality descends from \cite[Proposition 5.4]{Di-Pietro.Droniou:17} (cf. Remark 5.5 therein):
For all $\u v_h  \in \dU{h,0}{k}$,
  \begin{equation}\label{eq:discrete.Poincare}
    \|  v_h \|_{L^\sob(\Omega)} \lesssim \|   \u v_h \|_{1,\sob,h}.
  \end{equation}
By virtue of this inequality, $\| {\cdot} \|_{1,\sob,h}$ is a norm on $\dU{h,0}{k}$ (reason as in \cite[Corollary 2.16]{Di-Pietro.Droniou:20}).

\subsection{Reconstructions}

For all $T \in \T_h$, we define the \emph{local gradient reconstruction} $\dgrad{k}{T} : \dU{T}{k} \to \Poly^{k}(T)^d$ such that, for all $\u v_T \in \dU{T}{k}$,
\begin{equation}\label{eq:G}
  \displaystyle\int_T \dgrad{k}{T} \u v_T \cdot \b\tau = \int_T \GRAD v_T \cdot \b\tau + \sum_{F \in \F_T} \int_F (v_F-v_T)~(\b\tau \cdot \b n_{TF})\qquad \forall \b\tau \in  \Poly^{k}(T)^d.
\end{equation}
By design, the following relation holds (see \cite[Section 7.2.5]{Di-Pietro.Droniou:20}):
For all $v\in W^{1,1}(T)$,
\begin{equation}\label{eq:G:proj}
  \dgrad{k}{T} (\I{T}{k} v) = \PROJ{T}{k}(\GRAD v).
\end{equation}
The \emph{local potential reconstruction} $\drec{k+1}{T} : \dU{T}{k} \to \Poly^{k+1}(T)$ is such that, for all $\u v_T \in \dU{T}{k}$,
\begin{equation}\label{eq:rT}
\text{
  $\displaystyle\int_T (\GRAD \drec{k+1}{T} \u v_T - \dgrad{k}{T} \u v_T) \cdot \GRAD w = 0$ for all $w \in  \Poly^{k+1}(T)$
  and
  $\int_T \drec{k+1}{T} \u v_T = \int_T v_T.$
  }
\end{equation}
Composed with the local interpolator, this reconstruction commutes with the elliptic projector; see \cite[Sections 1.3 and 2.1.1--2.1.3]{Di-Pietro.Droniou:20}.

\subsection{Discrete diffusion function}

The \emph{discrete diffusion function} $\cst{a}_h : \dU{h}{k} \times \dU{h}{k} \to \R$, discretizing the function $a$ defined by \eqref{eq:a.b}, is such that, for all $\u v_h,\u w_h \in \dU{h}{k}$,
\begin{equation}\label{eq:ah}
  \cst{a}_h(\u w_h, \u v_h) \coloneqq
    \sum_{T \in \T_h}\left(\int_T \stress(\cdot,\dgrad{k}{T} \u w_T)\cdot \dgrad{k}{T} \u v_T+\cst{s}_T(\u w_T,\u v_T)\right).
\end{equation}
Above, for all $T\in\T_h$, $\cst{s}_T:\dU{T}{k}\times\dU{T}{k}\to\mathbb{R}$ is a local stabilization function.
To state the assumptions on this function, we introduce the mesh skeleton $\partial\mathcal{M}_h\coloneq\bigcup_{F\in\F_h}\overline{F}$ and set
 \begin{equation}
 \begin{aligned}
 L^p(\partial\mathcal{M}_h)&\coloneq \left\{\mu\ : \ \partial\mathcal{M}_h\to\R\,:\,\mu_{|F}\in L^p(F)\quad\forall F\in\F_h\right\},\\
\|\mu\|_{L^p(\partial\mathcal{M}_h)}&\coloneq \left(\sum_{T\in\T_h}h_T\sum_{F\in\F_T}\|\mu_{|F}\|_{L^p(F)}^p\right)^{\frac1p}.
 \end{aligned}
 \end{equation}
 \begin{assumption}[Local stabilization functions]\label{ass:sT}
   There exists $\zeta \in L^p(\partial \mathcal{M}_h;[0,+\infty))$ such that, for all $T \in \T_h$ and all $\u v_T,\u w_T \in \dU{T}{k}$,
   \begin{equation}
     \cst{s}_T(\u w_T,\u v_T) \coloneq h_T\int_{\partial T}\stab_T(\cdot,\dfbres{k}{\partial T}\u w_T)~\dfbres{k}{\partial T}\u v_T,
   \end{equation}
   where $\stab_T : \partial T \times \R \rightarrow \R$ is a measurable function satisfying, for all $v,w \in \R$ and almost every $\b x \in \partial T$,
   \begin{subequations}\label{eq:S:holder.continuity.strong.monotonicity}
     \begin{align} 
       |\stab_T(\b x,w)-\stab_T(\b x,v)| &\lesssim \sigma_\cst{hc}\left(\zeta(\b x)^p+|w|^p+|v|^p\right)^\frac{p-2}{p}| w-v |,\label{eq:S:holder.continuity}  \\
       \left(\stab_T(\b x,w)-\stab_T(\b x,v)\right) \left(w-v\right) &\gtrsim \sigma_\cst{sm}\left(\zeta(\b x)^p+|w|^p+|v|^p\right)^\frac{p-2}{p}|w-v|^{2},\label{eq:S:strong.monotonicity}\\
       \stab_T(\b x,0) &= 0, \label{eq:S:zero}
     \end{align}
   \end{subequations}
   while the boundary residual operator $\dfbres{k}{\partial T} : \dU{T}{k} \rightarrow L^p(\partial T)$ is such that, for all $\u v_T \in \dU{T}{k}$,
   \begin{equation}\label{eq:dfbres}
     (\dfbres{k}{\partial T} \u v_T)\res{F} \coloneq
     \frac{1}{h_T}\left[
       \proj{F}{k}(\drec{k+1}{T} \u v_T-v_F)-\proj{T}{k}(\drec{k+1}{T} \u v_T-v_T)
       \right]
     \qquad\forall F\in\F_T
   \end{equation}
   with potential reconstruction $\drec{k+1}{T}$ defined by \eqref{eq:rT}.
 \end{assumption}
 \begin{example}[Stabilization function]
   Local stabilization functions that match Assumption \ref{ass:sT} can be obtained setting, for all $T \in \T_h$, all $w \in \R$, and all $\b x \in \partial T$,
   \begin{equation}\label{eq:sT}
     \stab_T(\b x,w) \coloneqq \gamma_T\left(\zeta(\b x)^p+|w|^p\right)^\frac{\sob-2}{p}w,
   \end{equation}
   with $\gamma_T \in [\sigma_\cst{sm},\sigma_\cst{hc}]$ (see \eqref{eq:power-framed:constants.bound}).
   It can be checked that $\stab_T$ is a non-degenerate $p$-power-framed function satisfying \eqref{eq:S:holder.continuity.strong.monotonicity}; see \cite[Appendix A]{Botti.Castanon-Quiroz.ea:20} for a proof.
 \end{example}
 Leveraging the results in \cite{Di-Pietro.Droniou:17*1}, and additionally using $h_F \simeq h_T$ for all $F \in \F_T$, it can be checked that, for all $q \in (1,+\infty)$,
 \begin{equation}\label{eq:sT:stability.boundedness}
   \| \dgrad{k}{T}\u v_T \|_{L^q(T)^d}^q + h_T\|\dfbres{k}{\partial T}\u v_T\|_{L^q(\partial T)}^q
   \simeq \|\u v_T\|_{1,q,T}^q\qquad\forall\u v_T\in\dU{T}{k}.
 \end{equation}
 Additionally, $\dfbres{k}{\partial T}$ is polynomially consistent, i.e.,
 \begin{equation}\label{eq:sT:polynomial.consistency}
   \dfbres{k}{\partial T}(\I{T}{k} w) = 0\qquad\forall w\in\Poly^{k+1}(T).
 \end{equation}

 \subsection{Discrete problem}

 The discrete problem reads: Find $\u u_h \in \dU{h,0}{k}$ such that
 \begin{equation}
   \label{eq:lr.discrete} \cst{a}_h(\u u_h,\u v_h) = \displaystyle\int_\Omega f v_h \qquad \forall \u v_h \in \dU{h,0}{k}.
 \end{equation}

%------------------------------------------------------------------------------%

\section{Error analysis}\label{sec:error.estimate}

In this section, after establishing a stability result for the discrete function $\cst{a}_h$, we prove the error estimate that constitutes the main result of this paper.

\subsection{Strong monotonicity of the discrete diffusion function}

We recall the following inequality between sums of power (see \cite[Eq. (15)]{Botti.Castanon-Quiroz.ea:20}):
Let an integer $n\ge 1$ and a real number $m \in (0,+\infty)$ be given. Then, for all $a_1,\ldots,a_n \in (0,+\infty) $, we have
\begin{equation}\label{eq:sum-power}
  n^{-(m-1)^\ominus}\sum_{i=1}^n a_i^m \le \left(\sum_{i=1}^n a_i\right)^m  \le n^{(m-1)^\oplus}\sum_{i=1}^n a_i^m.
\end{equation}

\begin{lemma}[Strong monotonicity of $\cst{a}_h$]\label{lem:ah:holder.continuity.strong.monotonicity}
  For all $\u v_h, \u w_h \in \dU{h}{k}$, setting $\u e_h\coloneq \u v_h - \u w_h$, it holds
  \begin{equation}\label{eq:ah:strong.monotonicity}
    \begin{aligned}
    \| \u e_h \|_{1,\sob,h}^2
    &\lesssim \sigma_\cst{sm}^{-1}\left(\|\delta\|_{L^\sob(\Omega)}^\sob\!+\|\zeta\|_{L^\sob(\partial\mathcal{M}_h)}^\sob\!+\| \u v_h \|_{1,\sob,h}^\sob\!+\| \u w_h \|_{1,\sob,h}^\sob\right)^\frac{2-\sob}{\sob}
    \\
    &\qquad\times\left(\cst{a}_h(\u v_h,\u e_h)-\cst{a}_h(\u w_h,\u e_h)\right).
    \end{aligned}
  \end{equation}
\end{lemma}

\begin{proof}
  Let $T \in \T_h$. Using the strong monotonicity \eqref{eq:power-framed:s.strong.monotonicity} of $\stress$ and the $\left(\frac{2}{2-\sob},\frac{2}{\sob}\right)$-H\"{o}lder inequality, we get
  \begin{equation}\label{eq:ah:sm:1}
    \begin{aligned}
      &\sigma_\cst{sm}^\frac{\sob}{2}\| \dgrad{k}{T} \u e_T\|_{L^\sob(T)^d}^\sob \\
      &\leq \int_T \Big(\delta^\sob+|\dgrad{k}{T}\u v_T|^\sob+|\dgrad{k}{T}\u w_T|^\sob\Big)^{\frac{2-\sob}{2}}\left[
        \Big(\stress(\cdot,\dgrad{k}{T}\u v_T)-\stress(\cdot,\dgrad{k}{T}\u w_T)\Big)\cdot\dgrad{k}{T} \u e_T
        \right]^\frac{\sob}{2}\\
      &\le \Big(\|\delta\|_{L^\sob(T)}^\sob+\| \dgrad{k}{T} \u v_T\|_{L^\sob(T)^d}^\sob+\| \dgrad{k}{T} \u w_T\|_{L^\sob(T)^d}^\sob\Big)^\frac{2-\sob}{2}
      \\
      &\qquad
      \times\left[
        \int_T \left(
        \stress(\cdot,\dgrad{k}{T}\u v_T)-\stress(\cdot,\dgrad{k}{T}\u w_T)
        \right)\cdot\dgrad{k}{T} \u e_T
        \right]^\frac{\sob}{2}\\
      &\lesssim \left(\|\delta\|_{L^\sob(T)}^\sob+\| \u v_T \|_{1,\sob,T}^\sob+\| \u w_T \|_{1,\sob,T}^\sob\right)^\frac{2-\sob}{2}
      \\
      &\qquad\times
      \left[
      \int_T \left(
      \stress(\cdot,\dgrad{k}{T}\u v_T)-\stress(\cdot,\dgrad{k}{T}\u w_T)
      \right)\cdot\dgrad{k}{T} \u e_T
      \right]^\frac{\sob}{2},
    \end{aligned}
  \end{equation}
  where the conclusion follows from the seminorm equivalence \eqref{eq:sT:stability.boundedness}.
  Similarly, the strong monotonicity \eqref{eq:S:strong.monotonicity} of $\stab_T$ followed by the same reasoning as above yields,
  \begin{equation}\label{eq:ah:sm:2}
  \begin{aligned}
    \sigma_\cst{sm}^\frac{\sob}{2}h_T\|\Delta_{\partial T}^k\u e_T\|_{L^p(\F_T)}^p
    &\lesssim \Big(h_T\|\zeta\|_{L^\sob(\partial T)}^\sob+\| \u v_T \|_{1,\sob,T}^\sob+\| \u w_T \|_{1,\sob,T}^\sob\Big)^\frac{2-\sob}{2}
    \\
    &\qquad\times
    \Big(
    \cst{s}_T(\u v_T,\u e_T)-\cst{s}_T(\u w_T,\u e_T)
    \Big)^\frac{\sob}{2}.
  \end{aligned}
  \end{equation}
  Combining the norm equivalence \eqref{eq:sT:stability.boundedness} with \eqref{eq:ah:sm:1} and \eqref{eq:ah:sm:2} and using \eqref{eq:sum-power} yields
  \begin{align*}
    \sigma_\cst{sm}^\frac{\sob}{2}\| \u e_T \|_{1,\sob,T}^\sob    
    &\lesssim \Big(\|\delta\|_{L^\sob(T)}^\sob+h_T\|\zeta\|_{L^\sob(\partial T)}^\sob+\| \u v_T \|_{1,\sob,T}^\sob+\| \u w_T \|_{1,\sob,T}^\sob\Big)^\frac{2-\sob}{2}
    \\
    &\qquad\times
    \Big(
    \cst{a}_T(\u v_T,\u e_T)-\cst{a}_T(\u w_T,\u e_T)
    \Big)^\frac{\sob}{2}.
  \end{align*}
  Summing over $T \in \T_h$, applying the discrete $\left(\frac{2}{2-p},\frac{2}{p}\right)$-H\"older inequality, and raising to the power $\frac{2}{\sob}$ yields \eqref{eq:ah:strong.monotonicity}.
\end{proof}

\begin{remark}[Well-posedness and a priori estimate]\label{thm:well-posedness}
  Using standard techniques (cf. \cite[Theorem 4.5]{Di-Pietro.Droniou:17}, see also \cite[Theorem 17]{Botti.Castanon-Quiroz.ea:20}), it can be proved that there exists a unique solution $\u u_h \in \dU{h,0}{k}$ to the discrete problem \eqref{eq:lr.discrete}.
  Additionally, it can be shown in a similar way as for the continuous case (cf. Proposition \ref{prop:a-priori}) that the following a priori bound holds:
  \begin{multline}\label{eq:discrete.solution:bounds}
    \| \u u_h \|_{1,\sob,h}
    \lesssim \left(\sigma_\cst{sm}^{-1}\| f \|_{L^{\sob'}(\Omega)}\right)^\frac{1}{\sob-1} \\
    \\
    +\min\!\left(\left(\|\delta\|_{L^\sob(\T_h)}^\sob{+}\|\zeta\|_{L^\sob(\partial\mathcal{M}_h)}^\sob\right)^\frac{1}{p}\!;\sigma_\cst{sm}^{-1}\left(\|\delta\|_{L^\sob(\T_h)}^\sob{+}\|\zeta\|_{L^\sob(\partial\mathcal{M}_h)}^\sob\right)^\frac{2-\sob}{\sob}\!\!\| f \|_{L^{\sob'}(\Omega)}\right).
  \end{multline}
\end{remark}

\subsection{Error estimate}

\begin{theorem}[Error estimate]\label{thm:error.estimate}
  Let $u \in U$ and $\u u_h \in \dU{h,0}{k}$ solve \eqref{eq:lr.weak} and \eqref{eq:lr.discrete}, respectively. Assume $u \in W^{k+2,\sob}(\T_h)$ and $\stress(\cdot,\GRAD u) \in W^{1,\sob'}(\Omega)^d \cap W^{k+1,\sob'}(\T_h)^d$.
  Then, under Assumptions \ref{ass:stress} and \ref{ass:sT},
  \begin{multline}\label{eq:error.estimate.original}
    \| \u u_h - \I{h}{k} u \|_{1,\sob,h} 
    \lesssim
    \mathcal N_{f}h^{k+1}|\stress(\cdot, \GRAD u)|_{W^{k+1,\sob'}(\T_h)^d}
    \\  
    + \mathcal N_{f}\sigma_\cst{hc}\left[\sum_{T\in\T_h}\left(\min\left(\eta_T;1\right)^{2-p}h_T^{(k+1)(p-1)}|u|_{W^{k+2,p}(T)}^{p-1}\right)^{p'}\right]^\frac{1}{p'},
  \end{multline}
  where, for all $T \in \T_h$, defining $\mathfrak D_T \coloneq \min\!\big(\essinf_{\b x \in T}\left(\delta(\b x)+|\GRAD u(\b x)|\right);\essinf_{\b x \in \partial T}\zeta(\b x)\big)$, we have set
  \begin{equation}\label{eq:eta.T}
\eta_T \coloneq \dfrac{h_T^{k+1}|u|_{W^{k+2,p}(T)}}{|T|^\frac{1}{p}\mathfrak D_T}
  \end{equation}
 with the convention that $\eta_T=+\infty$ if $\mathfrak D_T=0 < |u|_{W^{k+2,p}(T)}$, and $\eta_T=0$ if $\mathfrak D_T=|u|_{W^{k+2,p}(T)}=0$, and where
  \[
    \mathcal N_{f} \coloneqq \sigma_\cst{sm}^{-1}\bigg[\|\delta\|_{L^\sob(\Omega)}+\|\zeta\|_{L^\sob(\partial\mathcal{M}_h)}+\left(\sigma_\cst{sm}^{-1}\| f \|_{L^{\sob'}(\Omega)}\right)^\frac{1}{\sob-1}\bigg]^{2-\sob}\!.
  \]
\end{theorem}

 \begin{remark}[Convergence rates]\label{rem:ocv}
 For any $T\in\T_h$, the local \emph{flux degeneracy parameter} $\mathfrak D_T$ which appears in \eqref{eq:eta.T} is a measure of the local degeneracy of the flux and the stabilization function: the closer it is to zero, the more degenerate the model is.
 The dimensionless number $\eta_T$ defined in \eqref{eq:eta.T} determines the convergence rate of the contribution to the approximation error stemming from $T$.
  If $\eta_T \ge 1$ (\emph{locally degenerate case}), then the element $T$ contributes to the error with a term in $h_T^{(k+1)(p-1)}$.
  If $\eta_T \le h_T^{k+1}|u|_{W^{k+2,p}(T)}|T|^{-\frac{1}{p}}$, i.e. $\mathfrak D_T \ge 1$ (\emph{locally non-degenerate case}), the contribution to the error is in $h_T^{k+1}$. 
  The case $\eta_T\in (h_T^{k+1}|u|_{W^{k+2,p}(T)}|T|^{-\frac{1}{p}},1)$ corresponds to intermediate rates of convergence. 

At the global level, defining the number $\eta_h \coloneq \max_{T\in\T_h}\eta_T$, the bound $h_T \le h$ together with the error estimate \eqref{eq:error.estimate.original} yields
 \label{rem:ocv.global}
  \begin{equation}\label{eq:error.estimate.inf}
    \begin{aligned}
    \| \u u_h - \I{h}{k} u \|_{1,\sob,h} 
    \lesssim
    \mathcal N_{f}\bigg(
    &h^{k+1}|\stress(\cdot, \GRAD u)|_{W^{k+1,\sob'}(\T_h)^d}
    \\
    &+\sigma_\cst{hc}\min\left(\eta_h;1\right)^{2-p}h^{(k+1)(p-1)}|u|_{W^{k+2,p}(\T_h)}^{p-1}
    \bigg).
      \end{aligned}
  \end{equation}
  As a consequence, if $\eta_h \ge 1$ (\emph{globally degenerate case}), then the convergence rate is $(k+1)(p-1)$. 
  If $\eta_h \le h^{k+1}|u|_{W^{k+2,p}(\T_h)}$ (\emph{globally non-degenerate case}), the convergence rate is $k+1$. 
  Finally, the case $\eta_h\in (h^{k+1}|u|_{W^{k+2,p}(\T_h)},1)$ corresponds to intermediate rates of convergence.
  This is the finest global estimate that can be obtained from the local one. However, for practical purposes, if $u \in W^{k+2,\infty}(\T_h)$ then we can replace $\eta_h$ in \eqref{eq:error.estimate.inf} by the larger number
    \begin{equation}\label{eq:eta.h}
    \tilde\eta_h \coloneq \frac{|u|_{W^{k+2,\infty}(\T_h)}h^{k+1}}{\min_{T\in\T_h}\mathfrak D_T} =  \frac{|u|_{W^{k+2,\infty}(\T_h)}h^{k+1}}{\min\left(\essinf_{\Omega}\left(\delta+|\GRAD u|\right);\essinf_{\partial\mathcal M_h}\zeta\right)},
  \end{equation}
  with the same convention as above regarding fractions $C/0$ and $0/0$.
The convergence rate will result from the position of $\tilde\eta_h$ with respect to $h^{k+1}|u|_{W^{k+2,\infty}(\T_h)}$ and $1$, and $\tilde\eta_h \le h^{k+1}|u|_{W^{k+2,\infty}(\T_h)}$ (non-degenerate case) is equivalent to $\min\left(\essinf_{\Omega}\left(\delta+|\GRAD u|\right);\essinf_{\partial\mathcal M_h}\zeta\right) \ge 1$, which is consistent with the local requirement stated above.
\end{remark}

\begin{proof}[Proof of Theorem \ref{thm:error.estimate}]
  Define the consistency error as the linear form $\mathcal E_h: \dU{h}{k} \to \R$ such that, for all $\u v_h \in \dU{h}{k}$,
  \begin{equation}\label{eq:Eah}
    \mathcal E_h(\u v_h) \coloneqq \int_\Omega \DIV \stress(\cdot,\GRAD u)~v_h + \cst{a}_h(\I{h}{k} u,\u v_h).
  \end{equation}
  Let, for the sake of brevity, $\u{\hat u}_h \coloneqq \I{h}{k} u$ and $\u e_h \coloneqq \u u_h-\u{\hat u}_h \in \dU{h,0}{k}$. 
  \medskip\\
  (i) \textit{Estimate of the consistency error.}
  Expanding $\cst{a}_h$ according to its definition \eqref{eq:ah} in the expression \eqref{eq:Eah} of $\mathcal E_h$, inserting $\sum_{T\in\T_h}\left(\int_T\PROJ{T}{k}\stress(\cdot,\GRAD u) \cdot \dgrad{k}{T}\u e_T-\int_T\stress(\cdot,\GRAD u) \cdot \dgrad{k}{T}\u e_T\right)=0$ (the equality is a consequence of the definition of $\PROJ{T}{k}$), and rearranging, we obtain
  \begin{multline}\label{eq:consistency:ah:EJ}
    \mathcal E_h(\u e_h)
    =
    \underbrace{%
      \int_\Omega \DIV \stress(\cdot,\GRAD u) ~ e_h
      + \sum_{T\in\T_h}\int_T \PROJ{T}{k}\stress(\cdot,\GRAD u) \cdot \dgrad{k}{T}\u e_T
    }_{\mathcal T_1}
    \\
    + \underbrace{%
      \sum_{T\in\T_h}\int_T\left[ \stress(\cdot,\dgrad{k}{T} \u{\hat u}_T) - \stress(\cdot,\GRAD u)\right] \cdot \dgrad{k}{T}\u e_T
    }_{\mathcal T_2}
    + \underbrace{\sum_{T \in \T_h}\cst{s}_T(\u{\hat u}_T,\u e_T)}_{\mathcal T_3}.
  \end{multline}
  We proceed to estimate the terms in the right-hand side.
  \smallskip
  
  For the first term, we start by noticing that
  \begin{equation}\label{eq:consistency:ah:null}
    \sum_{T \in \T_h}\sum_{F \in \F_T} \int_F  e_F \left(\stress(\cdot,\GRAD u)\cdot\b n_{TF}\right) = 0
  \end{equation}
  as a consequence of the continuity of the normal trace of $\stress(\cdot,\GRAD u)$ together with the single-valuedness of $e_F$ across each interface $F\in\Fi$ and the fact that $e_F= 0$ for every boundary face $F\in\Fb$ (see \cite[Corollary 1.19]{Di-Pietro.Droniou:20}).
  Using an element by element integration by parts on the first term of $\mathcal T_1$ along with the definition \eqref{eq:G} of $\dgrad{k}{T}$, we can write
  \[
  \begin{aligned}
    \mathcal T_1
    &= \cancel{\sum_{T\in\T_h}\int_T \left[\PROJ{T}{k}\stress(\cdot,\GRAD u)- \stress(\cdot,\GRAD u)\right] \cdot \GRAD e_T} \\
    &\qquad + \sum_{T \in \T_h}\sum_{F \in \F_T} \left[
      \int_F  (e_F-e_T)\big(\PROJ{T}{k}\stress(\cdot,\GRAD u) \cdot \b n_{TF}\big)
      + \int_F e_T\left(\stress(\cdot,\GRAD u)\cdot \b n_{TF}\right)
      \right] \\
    &=  \sum_{T \in \T_h}\sum_{F \in \F_T} \int_F  (e_F-e_T)\left[\PROJ{T}{k}\stress(\cdot,\GRAD u)-\stress(\cdot,\GRAD u)\right]\cdot \b n_{TF},
  \end{aligned}
  \]
  where we have used the definition of $\PROJ{T}{k}$ together with the fact that $\GRAD e_T \in \Poly^{k-1}(T)^d \subset \Poly^{k}(T)^d$ to cancel the term in the first line,
  and we have inserted \eqref{eq:consistency:ah:null} and rearranged to conclude.
  H\"{o}lder inequalities give
  \begin{equation}\label{eq:consistency:ah:T1}
    \begin{aligned}
      \left|\mathcal T_1\right|
      &\lesssim \left(\displaystyle\sum_{T \in \T_h}h_T \|\stress(\cdot,\GRAD u)- \PROJ{T}{k}\stress(\cdot,\GRAD u) \|_{L^{\sob'}(\partial T)^d}^{\sob'} \right)^\frac{1}{\sob'}
      \\
      &\qquad\times\left(\sum_{T \in \T_h}\sum_{F \in \F_T}h_F^{1-\sob}\| e_F-e_T\|_{L^{\sob}(F)}^{\sob}\right)^\frac{1}{\sob}
      \\
      &\lesssim h^{k+1} |\stress(\cdot, \GRAD u)|_{W^{k+1,\sob'}(\T_h)^d}\| \u e_h \|_{1,\sob,h},
    \end{aligned}
  \end{equation}
  where we have used the $(k+1,\sob')$-trace approximation properties \eqref{eq:proj:app:F} of $\PROJ{T}{k}$ along with $h_T \le h$ for the first factor, and the definition \eqref{eq:norm.epsilon.r} of $\|{\cdot}\|_{1,\sob,h}$ for the second.
  \smallskip
      
   We move on to the next term $\mathcal T_2$. Let an element $T \in \T_h$ be fixed.
If $\eta_T \ge 1$, using the $(p',p)$-H\"older inequality together with the equivalence \eqref{eq:sT:stability.boundedness}, we obtain
\begin{equation}\label{eq:consistency:ah:T2:0}
  \begin{aligned}
    &\left|\int_T\left[ \stress(\cdot,\dgrad{k}{T} \u{\hat u}_T) - \stress(\cdot,\GRAD u)\right] \cdot \dgrad{k}{T}\u e_T\right|
    \\
    &\qquad\le \|\stress(\cdot,\dgrad{k}{T} \u{\hat u}_T)- \stress(\cdot,\GRAD u) \|_{L^{\sob'}(T)^d}\| \u e_T \|_{1,\sob,T} \\
    &\qquad \le \sigma_\cst{hc} \left\| \big(\delta+|\PROJ{T}{k}(\GRAD u) |  +  | \GRAD u |\big)^{\sob-2}| \PROJ{T}{k}(\GRAD u) - \GRAD u |\right\|_{L^{\sob'}(T)}\| \u e_T \|_{1,\sob,T} \\
    &\qquad \le \sigma_\cst{hc}  \| \PROJ{T}{k}(\GRAD u) - \GRAD u \|_{L^{\sob}(T)^d}^{p-1}\| \u e_T \|_{1,\sob,T}\\
    &\qquad \lesssim \sigma_\cst{hc}h_T^{(k+1)(p-1)}|u|_{W^{k+2,p}(T)}^{p-1}\|\u e_T\|_{1,\sob,T} \\
    &\qquad =  \sigma_\cst{hc} \min\left(\eta_T;1\right)^{2-p}h_T^{(k+1)(p-1)}|u|_{W^{k+2,p}(T)}^{p-1} \| \u e_T \|_{1,\sob,T},
  \end{aligned}
\end{equation}
where we have used the continuity \eqref{eq:power-framed:s.holder.continuity} of $\stress$ together with the commutation property \eqref{eq:G:proj} of the discrete gradient and \eqref{eq:sum-power} in the second bound,
inequality \eqref{eq:prolongement} with $x = \PROJ{T}{k}(\GRAD u)$, $y =  \GRAD u$, and $\alpha = \delta$ in the third bound, and the $(k+1,p,0)$-approximation properties of $\proj{T}{k}$ to conclude.
  
On the other hand, if $\eta_T < 1$ then, using the $(p,p')$-H\"older inequality together with the boundedness \eqref{eq:sT:stability.boundedness}, we infer
\begin{equation}\label{eq:consistency:ah:T2:1}
      \begin{aligned}
    &\left|\int_T\left[ \stress(\cdot,\dgrad{k}{T} \u{\hat u}_T) {-} \stress(\cdot,\GRAD u)\right] \cdot \dgrad{k}{T}\u e_T\right|
    \le \|\stress(\cdot,\dgrad{k}{T} \u{\hat u}_T) {-} \stress(\cdot,\GRAD u) \|_{L^{\sob}(T)^d}\| \u e_T \|_{1,\sob',T} \\
    &\qquad \lesssim \sigma_\cst{hc} \left\| \big(\delta+|\PROJ{T}{k}(\GRAD u) |  +  | \GRAD u |\big)^{\sob-2}| \PROJ{T}{k}(\GRAD u) - \GRAD u |\right\|_{L^{\sob}(T)}|T|^\frac{p-2}{p}\| \u e_T \|_{1,\sob,T} \\
    &\qquad \le \sigma_\cst{hc}  \mathfrak D_T^{\sob-2}|T|^\frac{p-2}{p}\| \PROJ{T}{k}(\GRAD u) - \GRAD u \|_{L^{\sob}(T)^d}\| \u e_T \|_{1,\sob,T}\\
    &\qquad \lesssim \sigma_\cst{hc}\mathfrak D_T^{\sob-2}|T|^\frac{p-2}{p}h_T^{k+1}|u|_{W^{k+2,p}(T)}\|\u e_T\|_{1,\sob,T} \\
    &\qquad =  \sigma_\cst{hc} \min\left(\eta_T;1\right)^{2-p}h_T^{(k+1)(p-1)}|u|_{W^{k+2,p}(T)}^{p-1} \| \u e_T \|_{1,\sob,T},
        \end{aligned}
\end{equation}  
    where we passed to the second line as in \eqref{eq:consistency:ah:T2:0} additionally using the bound $\| \u e_T \|_{1,\sob',T} \lesssim |T|^\frac{p-2}{p}\| \u e_T \|_{1,\sob,T}$ (see \cite[Lemmas 5.1 and 5.2]{Di-Pietro.Droniou:17}), used in the third line the fact that $\R \ni x \mapsto x^{\sob-2} \in \R$ is non-increasing to infer that $\left(\delta+| \PROJ{T}{k}(\GRAD u) |  +  | \GRAD u |\right)^{\sob-2}\le\left(\delta  +  | \GRAD u |\right)^{\sob-2}\le \mathfrak D_T^{\sob-2}$ almost everywhere in $T$, and concluded as above.
  
  Gathering the estimates \eqref{eq:consistency:ah:T2:0} and \eqref{eq:consistency:ah:T2:1} and using a discrete H\"older inequality yields
  \begin{equation}\label{eq:consistency:ah:T2}
    \left|\mathcal T_2\right|
    \lesssim 
    \sigma_\cst{hc}\left[\sum_{T\in\T_h}\left(\min\left(\eta_T;1\right)^{2-p}h_T^{(k+1)(p-1)}|u|_{W^{k+2,p}(T)}^{p-1}\right)^{p'}\right]^\frac{1}{p'} \| \u e_h \|_{1,\sob,h}.
  \end{equation}
  \smallskip

Let us finally consider $\mathcal T_3$. Let $T \in \T_h$, set $\u{\check u}_T \coloneq \I{T}{k}(\proj{T}{k+1} u)$ for the sake of brevity, and observe that $\stab_T\big(\cdot,\dfbres{k}{\partial T}\u{\check u}_T\big) = 0$ thanks to the polynomial consistency \eqref{eq:sT:polynomial.consistency} of $\dfbres{k}{\partial T}$ and the property \eqref{eq:S:zero} of $\stab_T$.
  
If $\eta_T \ge 1$, using the $(p',p)$-H\"older inequality together with the boundedness property \eqref{eq:sT:stability.boundedness} (with $q = p$), we infer 
  \begin{equation}\label{eq:sT:consist:I:0}
  \begin{aligned}
    |\cst{s}_T(\u{\hat u}_T,\u e_T)|
    &\lesssim h_T^\frac{1}{p'}\left\|
    \stab_T\big(\cdot,\dfbres{k}{\partial T}\u{\hat u}_T\big)
    - \stab_T\big(\cdot,\dfbres{k}{\partial T}\u{\check u}_T\big)
    \right\|_{L^{p'}(\partial T)}\|\u e_T\|_{1,\sob,T} \\
    &\lesssim \sigma_\cst{hc}h_T^\frac{1}{p'}\left\|\left(\zeta+|\dfbres{k}{\partial T}\u{\hat u}_T|+|\dfbres{k}{\partial T}\u{\check u}_T|
    \right)^{p-2}\hspace{-0.5ex}\dfbres{k}{\partial T}(\u{\hat u}_T{-}\u{\check u}_T)\right\|_{L^{p'}(\partial T)}\hspace{-1ex}\|\u e_T\|_{1,\sob,T} \\
    &\le \sigma_\cst{hc}h_T^\frac{1}{p'}\|\dfbres{k}{\partial T}(\u{\hat u}_T-\u{\check u}_T)\|_{L^{p}(\partial T)}^{p-1}\|\u e_T\|_{1,\sob,T} \\
    &\lesssim \sigma_\cst{hc}h_T^{(k+1)(p-1)}|u|_{W^{k+2,p}(T)}^{p-1}\|\u e_T\|_{1,\sob,T} \\
    &=  \sigma_\cst{hc} \min\left(\eta_T;1\right)^{2-p}h_T^{(k+1)(p-1)}|u|_{W^{k+2,p}(T)}^{p-1} \| \u e_T \|_{1,\sob,T},
  \end{aligned}
  \end{equation}  
  where we have used the continuity \eqref{eq:S:holder.continuity} of $\stab_T$ together with \eqref{eq:sum-power} to pass to the second line, inequality \eqref{eq:prolongement} with $x = \dfbres{k}{\partial T}\u{\hat u}_T$ , $y = \dfbres{k}{\partial T} \u{\check u}_T$ and $\alpha = \zeta$ in the third line, and the boundedness \eqref{eq:sT:stability.boundedness} of $\dfbres{k}{\partial T}$ and \eqref{eq:I:boundedness} of $\I{T}{k}$ together with the $(k+1,p,1)$-approximation properties of $\proj{T}{k+1}$ to conclude by writing $h_T^\frac{1}{p}\|\dfbres{k}{\partial T}(\u{\hat u}_T-\u{\check u}_T)\|_{L^{p}(\partial T)} \lesssim \|\I{T}{k}(u-\proj{T}{k+1}u)\|_{1,p,T} \lesssim |u-\proj{T}{k+1}u|_{W^{1,p}(T)} \lesssim h_T^{k+1}|u|_{W^{k+2,p}(T)}$.
  
    Otherwise, $\eta_T < 1$ and using the $(p,p')$-H\"older inequality together with boundedness property \eqref{eq:sT:stability.boundedness} (with $q = p'$), we infer as above that
   \begin{equation}\label{eq:sT:consist:I:1}
  \begin{aligned}
    &|\cst{s}_T(\u{\hat u}_T,\u e_T)|
    \\
    &\quad\lesssim h_T^\frac{1}{p}\left\|
    \stab_T\big(\cdot,\dfbres{k}{\partial T}\u{\hat u}_T\big)
    - \stab_T\big(\cdot,\dfbres{k}{\partial T}\u{\check u}_T\big)
    \right\|_{L^{p}(\partial T)}\| \u e_T \|_{1,\sob',T} \\
    &\quad\lesssim \sigma_\cst{hc}h_T^\frac{1}{p}\left\|\left(\zeta+|\dfbres{k}{\partial T}\u{\hat u}_T|+|\dfbres{k}{\partial T}\u{\check u}_T|
    \right)^{p-2}\dfbres{k}{\partial T}(\u{\hat u}_T-\u{\check u}_T)\right\|_{L^{p}(\partial T)}|T|^\frac{p-2}{p}\| \u e_T \|_{1,\sob,T}  \\
    &\quad\le \sigma_\cst{hc}\mathfrak D_T^{p-2}|T|^\frac{p-2}{p}h_T^\frac{1}{p}\|\dfbres{k}{\partial T}(\u{\hat u}_T-\u{\check u}_T)\|_{L^{\sob}(\partial T)}\| \u e_T \|_{1,\sob,T}   \\
    &\quad\lesssim \sigma_\cst{hc}\mathfrak D_T^{p-2}|T|^\frac{p-2}{p}h_T^{k+1}|u|_{W^{k+2,p}(T)}\|\u e_T\|_{1,\sob,T} \\
    &\quad=  \sigma_\cst{hc} \min\left(\eta_T;1\right)^{2-p}h_T^{(k+1)(p-1)}|u|_{W^{k+2,p}(T)}^{p-1} \| \u e_T \|_{1,\sob,T},
  \end{aligned}
  \end{equation}   
   where the second inequality follows as before from the bound $\| \u e_T \|_{1,\sob',T} \lesssim |T|^\frac{p-2}{p}\| \u e_T \|_{1,\sob,T}$ (see \cite[Lemmas 5.1 and 5.2]{Di-Pietro.Droniou:17}),
   the third inequality is a consequence of the monotonicity of $\R \ni x \mapsto x^{\sob-2} \in \R$ that yields $\big(\zeta+|\dfbres{k}{\partial T}\u{\hat u}_T|  +  |\dfbres{k}{\partial T}\u{\check u}_T|\big)^{\sob-2}\le\zeta^{\sob-2}\le \mathfrak D_T^{\sob-2}$ almost everywhere in $\partial T$, and the conclusion is obtained as in \eqref{eq:sT:consist:I:0}.
      
  Following then the same reasoning that lead to \eqref{eq:consistency:ah:T2}, we obtain for the third term
  \begin{equation}\label{eq:consistency:ah:T3}
    \left|\mathcal T_3\right|
    \lesssim 
    \sigma_\cst{hc}\left[\sum_{T\in\T_h}\left(\min\left(\eta_T;1\right)^{2-p}h_T^{(k+1)(p-1)}|u|_{W^{k+2,p}(T)}^{p-1}\right)^{p'}\right]^\frac{1}{p'} \| \u e_h \|_{1,\sob,h}.
  \end{equation}

  Plugging the bounds \eqref{eq:consistency:ah:T1}, \eqref{eq:consistency:ah:T2}, and \eqref{eq:consistency:ah:T3} into \eqref{eq:consistency:ah:EJ} yields
  \begin{equation}\label{eq:consistency:ah}
    \begin{aligned}
    |\mathcal E_h(\u e_h)| &\lesssim
    h^{k+1}|\stress(\cdot, \GRAD u)|_{W^{k+1,\sob'}(\T_h)^d}\| \u e_h \|_{1,\sob,h}
    \\
    &\quad +
    \sigma_\cst{hc}\left[\sum_{T\in\T_h}\left(\min\left(\eta_T;1\right)^{2-p}h_T^{(k+1)(p-1)}|u|_{W^{k+2,p}(T)}^{p-1}\right)^{p'}\right]^\frac{1}{p'} \| \u e_h \|_{1,\sob,h}.
    \end{aligned}
  \end{equation}
  \medskip\\
  (ii) \textit{Error estimate.} 
  Using the strong monotonicity \eqref{eq:ah:strong.monotonicity} of $\cst{a}_h$, we get
  \begin{align}
    \| \u e_h \|_{1,\sob,h}^2
    &\lesssim \sigma_\cst{sm}^{-1}\left(
    \|\delta\|_{L^\sob(\Omega)}^\sob+\|\zeta\|_{L^\sob(\partial\mathcal{M}_h)}^\sob+\| \u u_h \|_{1,\sob,h}^\sob+\|\u{\hat u}_h\|_{1,\sob,h}^\sob
    \right)^\frac{2-\sob}{\sob}
    \\
    &\qquad\times\left[
    \cst{a}_h(\u u_h,\u e_h)-\cst{a}_h(\u{\hat u}_h,\u e_h)
    \right] \nonumber\\
    &\lesssim \mathcal N_{f}\left[
      \cst{a}_h(\u u_h,\u e_h)-\cst{a}_h(\u{\hat u}_h,\u e_h)
      \right],
  \label{eq:error.estimate:step2:eh0}
  \end{align}
  where we have used the a priori bound \eqref{eq:discrete.solution:bounds} on the discrete solution along with the boundedness \eqref{eq:I:boundedness} of the global interpolator, the a priori bound \eqref{eq:continuous.solution:bounds:uh} on the continuous solution, and \eqref{eq:sum-power} to conclude.
  Furthermore, using the equation \eqref{eq:lr.discrete} (with $\u v_h = \u e_h$), and the fact that $f = -\DIV \stress(\cdot,\GRAD u)$ almost everywhere in $\Omega$, we see that
  \begin{equation}\label{eq:error.estimate:step1:eh1}
    \cst{a}_h(\u u_h,\u e_h)-\cst{a}_h(\u{\hat u}_h,\u e_h)
    = \int_\Omega f e_h - \cst{a}_h(\u{\hat u}_h,\u e_h)
    = -\mathcal E_h(\u e_h).
  \end{equation}
  Hence, plugging \eqref{eq:error.estimate:step1:eh1} into \eqref{eq:error.estimate:step2:eh0}, recalling the bound \eqref{eq:consistency:ah} on the consistency error, and simplifying, \eqref{eq:error.estimate.original} follows.
\end{proof}

%------------------------------------------------------------------------------------%

\section{Numerical examples}\label{sec:num.res}

In this section, we give some numerical results to confirm Theorem \ref{thm:error.estimate}.
We consider the domain $\Omega=(0,1)^2$ and define $\stress$ as the Carreau--Yasuda law of Example \ref{ex:Carreau--Yasuda} with $\sob\in\{1.25,~\!1.5,~\!1.75\}$ and $\mu = a = 1$. The degeneracy parameter $\delta$ and the exact solution $u$ will depend on the considered case.
The stabilization functions are defined by \eqref{eq:sT} with $\zeta$ such that the local flux degeneracy number $\mathfrak D_T$ introduced in \eqref{eq:eta.T} is equal to the first argument of its min for all $T \in \T_h$, so that $\zeta$ does not influence the error estimates.
The function $f$ and the Dirichlet boundary condition are inferred from the exact solution.
In all cases, except for the non-homogeneous boundary condition, these solutions match the assumptions required in Theorem \ref{thm:error.estimate}.
We consider the HHO scheme for $k \in \{1,2,3\}$ on a triangular mesh family.

\subsection{Non-degenerate flux}\label{subsec:num.res:flux} 

We consider nonzero constant degeneracy parameters $\delta \in \{1,~\!0.1,~\!10^{-2},~\!5 \cdot 10^{-4}\}$, and the potential $u$ is given by
\[
  u(x,y) = \sin\left(\pi x\right)\sin\left(\pi y\right) \quad \forall (x,y) \in \Omega.
  \]
Thus, the dimensionless number $\tilde\eta_h$ defined in \eqref{eq:eta.h} satisfies
\begin{equation}\label{eq:eta}
\tilde\eta_h = \mu_h(\delta) \coloneq \frac{2^\frac{k-1}{2}\pi^{k}h^{k+1}}{\delta}.
\end{equation}
Therefore, we should observe a $(k+1)(p-1)$ pre-asymptotic order of convergence until the size of the mesh is small enough compared to $\delta$ so that the convergence rate switches to $k+1$ (see Remark \ref{rem:ocv}).

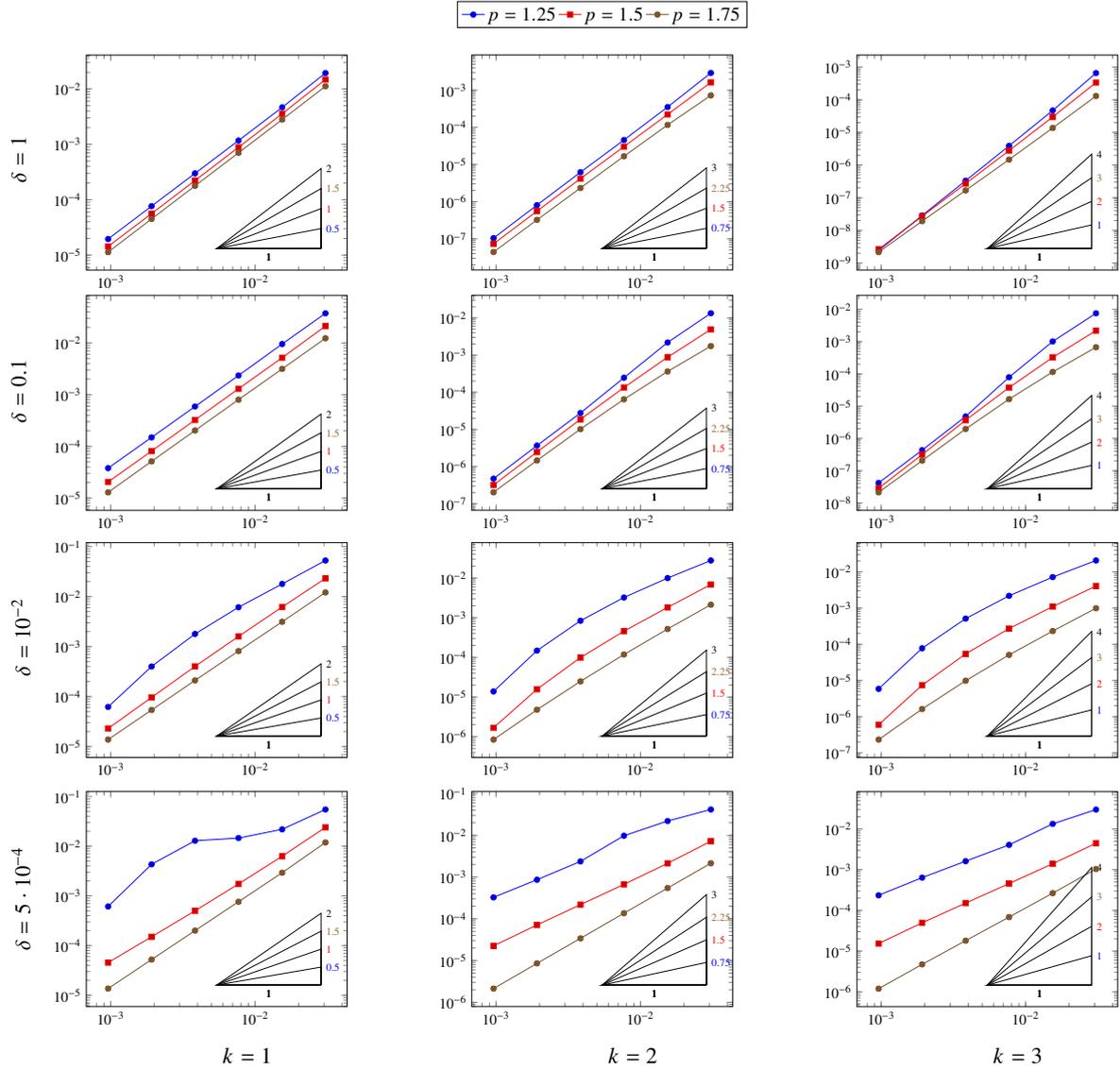
\begin{figure}
  \hspace{-5cm}\begin{center}
    %delta = 1
        \begin{minipage}[b]{0.05\columnwidth}
      \begin{tikzpicture}
      	\draw node[scale=0.8,rotate=90]{\hspace*{1.7cm} $\delta = 1$};
      \end{tikzpicture}
    \end{minipage}\hspace{-9mm}
    \hfill
    \begin{minipage}[b]{0.29\columnwidth}
      \begin{tikzpicture}[scale=0.53]
        \begin{loglogaxis}
          \addplot table[x=meshsize,y=err] {results/mesh1/k1_a1_p125_del1.txt};
          \addplot table[x=meshsize,y=err] {results/mesh1/k1_a1_p150_del1.txt};
          \addplot table[x=meshsize,y=err] {results/mesh1/k1_a1_p175_del1.txt};
          \logLogSlopeTriangle{0.90}{0.4}{0.1}{2}{black};
          \logLogSlopeTriangle{0.90}{0.4}{0.1}{0.5}{blue};
          \logLogSlopeTriangle{0.90}{0.4}{0.1}{1}{red};
          \logLogSlopeTriangle{0.90}{0.4}{0.1}{1.5}{darkbrown};                      
        \end{loglogaxis}
      \end{tikzpicture}
    \end{minipage}
    \hfill
    \begin{minipage}[b]{0.29\columnwidth}
      \begin{tikzpicture}[scale=0.53]
        \begin{loglogaxis}[
            legend style = { 
              at={(0.5,1.25)},
              anchor = north,
              legend columns=-1,
              font={\Large}
            }
          ]
          \addplot table[x=meshsize,y=err] {results/mesh1/k2_a1_p125_del1.txt};
          \addplot table[x=meshsize,y=err] {results/mesh1/k2_a1_p150_del1.txt};
          \addplot table[x=meshsize,y=err] {results/mesh1/k2_a1_p175_del1.txt};
          \logLogSlopeTriangle{0.90}{0.4}{0.1}{3}{black};
          \logLogSlopeTriangle{0.90}{0.4}{0.1}{0.75}{blue};
          \logLogSlopeTriangle{0.90}{0.4}{0.1}{1.5}{red};
          \logLogSlopeTriangle{0.90}{0.4}{0.1}{2.25}{darkbrown};
          \legend{$\sob=1.25$ , $\sob=1.5$ , $\sob=1.75$};
        \end{loglogaxis}
      \end{tikzpicture}
    \end{minipage}
    \hfill
    \begin{minipage}[b]{0.29\columnwidth}
      \begin{tikzpicture}[scale=0.53]
        \begin{loglogaxis}
          \addplot table[x=meshsize,y=err] {results/mesh1/k3_a1_p125_del1.txt};
          \addplot table[x=meshsize,y=err] {results/mesh1/k3_a1_p150_del1.txt};
          \addplot table[x=meshsize,y=err] {results/mesh1/k3_a1_p175_del1.txt};
          \logLogSlopeTriangle{0.90}{0.4}{0.1}{4}{black};
          \logLogSlopeTriangle{0.90}{0.4}{0.1}{1}{blue};
          \logLogSlopeTriangle{0.90}{0.4}{0.1}{2}{red};
          \logLogSlopeTriangle{0.90}{0.4}{0.1}{3}{darkbrown};
        \end{loglogaxis}
      \end{tikzpicture}
    \end{minipage}
    \\ %delta = 0.1
        \begin{minipage}[b]{0.05\columnwidth}
      \begin{tikzpicture}
      	\draw node[scale=0.8,rotate=90]{\hspace*{1.7cm} $\delta = 0.1$};
      \end{tikzpicture}
    \end{minipage}\hspace{-9mm}
    \hfill
    \begin{minipage}[b]{0.29\columnwidth}
      \begin{tikzpicture}[scale=0.53]
        \begin{loglogaxis}
          \addplot table[x=meshsize,y=err] {results/mesh1/k1_a1_p125_del01.txt};
          \addplot table[x=meshsize,y=err] {results/mesh1/k1_a1_p150_del01.txt};
          \addplot table[x=meshsize,y=err] {results/mesh1/k1_a1_p175_del01.txt};
          \logLogSlopeTriangle{0.90}{0.4}{0.1}{2}{black};
          \logLogSlopeTriangle{0.90}{0.4}{0.1}{0.5}{blue};
          \logLogSlopeTriangle{0.90}{0.4}{0.1}{1}{red};
          \logLogSlopeTriangle{0.90}{0.4}{0.1}{1.5}{darkbrown};         
%          \logLogSlopeTriangle{0.90}{0.4}{0.1}{0.5}{blue};
%          \logLogSlopeTriangle{0.90}{0.4}{0.1}{1}{red};
%          \logLogSlopeTriangle{0.90}{0.4}{0.1}{1}{darkbrown};                
        \end{loglogaxis}
      \end{tikzpicture}
    \end{minipage}
    \hfill
    \begin{minipage}[b]{0.29\columnwidth}
      \begin{tikzpicture}[scale=0.53]
        \begin{loglogaxis}[
            legend style = { 
              at={(0.5,1.25)},
              anchor = north,
              legend columns=-1,
              font={\Large}
            }
          ]
          \addplot table[x=meshsize,y=err] {results/mesh1/k2_a1_p125_del01.txt};
          \addplot table[x=meshsize,y=err] {results/mesh1/k2_a1_p150_del01.txt};
          \addplot table[x=meshsize,y=err] {results/mesh1/k2_a1_p175_del01.txt};
          \logLogSlopeTriangle{0.90}{0.4}{0.1}{3}{black};
          \logLogSlopeTriangle{0.90}{0.4}{0.1}{0.75}{blue};
          \logLogSlopeTriangle{0.90}{0.4}{0.1}{1.5}{red};
          \logLogSlopeTriangle{0.90}{0.4}{0.1}{2.25}{darkbrown};
        \end{loglogaxis}
      \end{tikzpicture}
    \end{minipage}
    \hfill
    \begin{minipage}[b]{0.29\columnwidth}
      \begin{tikzpicture}[scale=0.53]
        \begin{loglogaxis}
          \addplot table[x=meshsize,y=err] {results/mesh1/k3_a1_p125_del01.txt};
          \addplot table[x=meshsize,y=err] {results/mesh1/k3_a1_p150_del01.txt};
          \addplot table[x=meshsize,y=err] {results/mesh1/k3_a1_p175_del01.txt};
          \logLogSlopeTriangle{0.90}{0.4}{0.1}{4}{black};
          \logLogSlopeTriangle{0.90}{0.4}{0.1}{1}{blue};
          \logLogSlopeTriangle{0.90}{0.4}{0.1}{2}{red};
          \logLogSlopeTriangle{0.90}{0.4}{0.1}{3}{darkbrown};
        \end{loglogaxis}
      \end{tikzpicture}
    \end{minipage}
    \\ %delta = 1.e-2
        \begin{minipage}[b]{0.05\columnwidth}
      \begin{tikzpicture}
      	\draw node[scale=0.8,rotate=90]{\hspace*{1.7cm}$\delta = 10^{-2}$};
      \end{tikzpicture}
    \end{minipage}\hspace{-9mm}
    \hfill
        \begin{minipage}[b]{0.29\columnwidth}
      \begin{tikzpicture}[scale=0.53]
        \begin{loglogaxis}
          \addplot table[x=meshsize,y=err] {results/mesh1/k1_a1_p125_del001.txt};
          \addplot table[x=meshsize,y=err] {results/mesh1/k1_a1_p150_del001.txt};
          \addplot table[x=meshsize,y=err] {results/mesh1/k1_a1_p175_del001.txt};
          \logLogSlopeTriangle{0.90}{0.4}{0.1}{2}{black};
          \logLogSlopeTriangle{0.90}{0.4}{0.1}{0.5}{blue};
          \logLogSlopeTriangle{0.90}{0.4}{0.1}{1}{red};
          \logLogSlopeTriangle{0.90}{0.4}{0.1}{1.5}{darkbrown};  
%          \logLogSlopeTriangle{0.90}{0.4}{0.1}{0.75}{blue};
%          \logLogSlopeTriangle{0.90}{0.4}{0.1}{1}{red};
%          \logLogSlopeTriangle{0.90}{0.4}{0.1}{2.25}{darkbrown};
        \end{loglogaxis}
      \end{tikzpicture}
    \end{minipage}
    \hfill
        \begin{minipage}[b]{0.29\columnwidth}
      \begin{tikzpicture}[scale=0.53]
        \begin{loglogaxis}
          \addplot table[x=meshsize,y=err] {results/mesh1/k2_a1_p125_del001.txt};
          \addplot table[x=meshsize,y=err] {results/mesh1/k2_a1_p150_del001.txt};
          \addplot table[x=meshsize,y=err] {results/mesh1/k2_a1_p175_del001.txt};
          \logLogSlopeTriangle{0.90}{0.4}{0.1}{3}{black};
          \logLogSlopeTriangle{0.90}{0.4}{0.1}{0.75}{blue};
          \logLogSlopeTriangle{0.90}{0.4}{0.1}{1.5}{red};
          \logLogSlopeTriangle{0.90}{0.4}{0.1}{2.25}{darkbrown};
        \end{loglogaxis}
      \end{tikzpicture}
    \end{minipage}
    \hfill
        \begin{minipage}[b]{0.29\columnwidth}
      \begin{tikzpicture}[scale=0.53]
        \begin{loglogaxis}
          \addplot table[x=meshsize,y=err] {results/mesh1/k3_a1_p125_del001.txt};
          \addplot table[x=meshsize,y=err] {results/mesh1/k3_a1_p150_del001.txt};
          \addplot table[x=meshsize,y=err] {results/mesh1/k3_a1_p175_del001.txt};
          \logLogSlopeTriangle{0.90}{0.4}{0.1}{4}{black};
          \logLogSlopeTriangle{0.90}{0.4}{0.1}{1}{blue};
          \logLogSlopeTriangle{0.90}{0.4}{0.1}{2}{red};
          \logLogSlopeTriangle{0.90}{0.4}{0.1}{3}{darkbrown};
        \end{loglogaxis}
      \end{tikzpicture}
    \end{minipage}
       \\ %delta = 5.e-4
        \begin{minipage}[b]{0.05\columnwidth}
      \begin{tikzpicture}
      	\draw node[scale=0.8,rotate=90]{\hspace*{2cm}$\delta = 5\cdot 10^{-4}$};
      \end{tikzpicture}
    \end{minipage}\hspace{-9mm}
    \hfill
        \begin{minipage}[b]{0.29\columnwidth}
      \begin{tikzpicture}[scale=0.53]
        \begin{loglogaxis}
          \addplot table[x=meshsize,y=err] {results/mesh1/k1_a1_p125_del00005.txt};
          \addplot table[x=meshsize,y=err] {results/mesh1/k1_a1_p150_del00005.txt};
          \addplot table[x=meshsize,y=err] {results/mesh1/k1_a1_p175_del00005.txt};
          \logLogSlopeTriangle{0.90}{0.4}{0.1}{2}{black};
          \logLogSlopeTriangle{0.90}{0.4}{0.1}{0.5}{blue};
          \logLogSlopeTriangle{0.90}{0.4}{0.1}{1}{red};
          \logLogSlopeTriangle{0.90}{0.4}{0.1}{1.5}{darkbrown};  
%          \logLogSlopeTriangle{0.90}{0.4}{0.1}{0.75}{blue};
%          \logLogSlopeTriangle{0.90}{0.4}{0.1}{1}{red};
%          \logLogSlopeTriangle{0.90}{0.4}{0.1}{2.25}{darkbrown};
        \end{loglogaxis}
      \end{tikzpicture}
      \subcaption*{\hspace{8mm} $k=1$}
    \end{minipage}
    \hfill
        \begin{minipage}[b]{0.29\columnwidth}
      \begin{tikzpicture}[scale=0.53]
        \begin{loglogaxis}
          \addplot table[x=meshsize,y=err] {results/mesh1/k2_a1_p125_del00005.txt};
          \addplot table[x=meshsize,y=err] {results/mesh1/k2_a1_p150_del00005.txt};
          \addplot table[x=meshsize,y=err] {results/mesh1/k2_a1_p175_del00005.txt};
          \logLogSlopeTriangle{0.90}{0.4}{0.1}{3}{black};
          \logLogSlopeTriangle{0.90}{0.4}{0.1}{0.75}{blue};
          \logLogSlopeTriangle{0.90}{0.4}{0.1}{1.5}{red};
          \logLogSlopeTriangle{0.90}{0.4}{0.1}{2.25}{darkbrown};
        \end{loglogaxis}
      \end{tikzpicture}
      \subcaption*{\hspace{8mm} $k=2$}
    \end{minipage}
    \hfill
        \begin{minipage}[b]{0.29\columnwidth}
      \begin{tikzpicture}[scale=0.53]
        \begin{loglogaxis}
          \addplot table[x=meshsize,y=err] {results/mesh1/k3_a1_p125_del00005.txt};
          \addplot table[x=meshsize,y=err] {results/mesh1/k3_a1_p150_del00005.txt};
          \addplot table[x=meshsize,y=err] {results/mesh1/k3_a1_p175_del00005.txt};
          \logLogSlopeTriangle{0.90}{0.4}{0.1}{4}{black};
          \logLogSlopeTriangle{0.90}{0.4}{0.1}{1}{blue};
          \logLogSlopeTriangle{0.90}{0.4}{0.1}{2}{red};
          \logLogSlopeTriangle{0.90}{0.4}{0.1}{3}{darkbrown};
        \end{loglogaxis}
      \end{tikzpicture}
      \subcaption*{\hspace{8mm} $k=3$}
    \end{minipage}
  \end{center}
   \vspace{-5mm}
   \caption{Numerical results for the test case of Subsection \ref{subsec:num.res:flux}. The steeper slope (in black) indicates the $k+1$ convergence rate expected from Theorem \ref{thm:error.estimate} when the number $\delta$ is large enough compared to the mesh size. Otherwise the other slopes indicate the $(k+1)(p-1)$ convergence rate according to $p$.
   \label{tab:num.res:flux}}
\end{figure}

\begin{table}
\caption{Convergence rates for the test case of Subsection \ref{subsec:num.res:flux}. The bold numbers in each column correspond to the $(k+1)(p-1)$ \textasciitilde\ $(k+1)$ convergence rates.}
\label{table:num.res:flux}       % Give a unique label

\begin{center}
\begin{tabular}{c|ccc|ccc|ccccc}
\hline\noalign{\smallskip}
\multicolumn{10}{c}{$\delta = 1$}\\
\hline\noalign{\smallskip}
 $k$ & & 1 & & & 2 & & & 3 \\
\hline\noalign{\smallskip}
\diagbox[width=4em]{$h$}{$p$} & 1.25 & 1.5 & 1.75 & 1.25 & 1.5 & 1.75 & 1.25 & 1.5 & 1.75\\
\noalign{\smallskip}\hline\noalign{\smallskip}
3.07e-02 & \textbf{0.5 \textasciitilde\ 2} & \textbf{1 \textasciitilde\ 2} & \textbf{1.5 \textasciitilde\ 2} & \textbf{0.75 \textasciitilde\ 3} & \textbf{1.5 \textasciitilde\ 3} & \textbf{2.25 \textasciitilde\ 3} & \textbf{1 \textasciitilde\ 4} & \textbf{2 \textasciitilde\ 4} & \textbf{3 \textasciitilde\ 4} \\
1.54e-02 & 2.07 & 2.06 & 2.01 & 3.06 & 2.88 & 2.65 & 3.83 & 3.51 & 3.26 \\
7.68e-03 & 1.98 & 2.01 & 1.98 & 2.94 & 2.87 & 2.79 & 3.59 & 3.41 & 3.21 \\
3.84e-03 & 1.97 & 1.99 & 1.98 & 2.88 & 2.86 & 2.83 & 3.55 & 3.35 & 3.14 \\
1.92e-03 & 1.97 & 1.98 & 1.99 & 2.94 & 2.90 & 2.85 & 3.54 & 3.33 & 3.13 \\
9.60e-04 & 1.97 & 1.97 & 1.99 & 2.95 & 2.93 & 2.88 & 3.56 & 3.34 & 3.14 \\
\noalign{\smallskip}\hline
\multicolumn{10}{c}{}\\[-6pt]
\multicolumn{10}{c}{$\delta = 0.1$}\\
\hline\noalign{\smallskip}
 $k$ & & 1 & & & 2 & & & 3 \\
\hline\noalign{\smallskip}
\diagbox[width=4em]{$h$}{$p$} & 1.25 & 1.5 & 1.75 & 1.25 & 1.5 & 1.75 & 1.25 & 1.5 & 1.75\\
\noalign{\smallskip}\hline\noalign{\smallskip}
3.07e-02 & \textbf{0.5 \textasciitilde\ 2} & \textbf{1 \textasciitilde\ 2} & \textbf{1.5 \textasciitilde\ 2} & \textbf{0.75 \textasciitilde\ 3} & \textbf{1.5 \textasciitilde\ 3} & \textbf{2.25 \textasciitilde\ 3} & \textbf{1 \textasciitilde\ 4} & \textbf{2 \textasciitilde\ 4} & \textbf{3 \textasciitilde\ 4} \\
1.54e-02	&1.98	&2.04	&1.97	&2.63	&2.49	&2.28	&2.92	&2.77	&2.55\\
7.68e-03	&2.01	&1.99	&1.97	&3.14	&2.70	&2.48	&3.67	&3.10	&2.79\\
3.84e-03	&1.99	&2.00	&1.98	&3.15	&2.84	&2.66	&4.04	&3.35	&3.07\\
1.92e-03	&1.99	&2.00	&1.98	&2.92	&2.93	&2.80	&3.47	&3.54	&3.26\\
9.60e-04	&1.98	&1.98	&1.99	&2.96	&2.95	&2.85	&3.37	&3.46	&3.26\\
\noalign{\smallskip}\hline
\multicolumn{10}{c}{}\\[-6pt]
\multicolumn{10}{c}{$\delta = 10^{-2}$}\\
\hline\noalign{\smallskip}
 $k$ & & 1 & & & 2 & & & 3 \\
\hline\noalign{\smallskip}
\diagbox[width=4em]{$h$}{$p$} & 1.25 & 1.5 & 1.75 & 1.25 & 1.5 & 1.75 & 1.25 & 1.5 & 1.75\\
\noalign{\smallskip}\hline\noalign{\smallskip}
3.07e-02 & \textbf{0.5 \textasciitilde\ 2} & \textbf{1 \textasciitilde\ 2} & \textbf{1.5 \textasciitilde\ 2} & \textbf{0.75 \textasciitilde\ 3} & \textbf{1.5 \textasciitilde\ 3} & \textbf{2.25 \textasciitilde\ 3} & \textbf{1 \textasciitilde\ 4} & \textbf{2 \textasciitilde\ 4} & \textbf{3 \textasciitilde\ 4} \\
1.54e-02	&1.57	&1.91	&1.96	&1.48	&1.93	&2.04	&1.53	&1.89	&2.10\\
7.68e-03	&1.54	&1.94	&1.94	&1.62	&1.98	&2.13	&1.71	&2.01	&2.18\\
3.84e-03	&1.77	&1.99	&1.95	&1.95	&2.21	&2.25	&2.09	&2.32	&2.37\\
1.92e-03	&2.16	&2.06	&1.97	&2.50	&2.65	&2.38	&2.72	&2.87	&2.60\\
9.60e-04	&2.69	&2.07	&1.97	&3.42	&3.23	&2.50	&3.73	&3.63	&2.80\\
\noalign{\smallskip}\hline
\multicolumn{10}{c}{}\\[-6pt]
\multicolumn{10}{c}{$\delta = 5 \cdot 10^{-4}$}\\
\hline\noalign{\smallskip}
 $k$ & & 1 & & & 2 & & & 3 \\
\hline\noalign{\smallskip}
\diagbox[width=4em]{$h$}{$p$} & 1.25 & 1.5 & 1.75 & 1.25 & 1.5 & 1.75 & 1.25 & 1.5 & 1.75\\
\noalign{\smallskip}\hline\noalign{\smallskip}
3.07e-02 & \textbf{0.5 \textasciitilde\ 2} & \textbf{1 \textasciitilde\ 2} & \textbf{1.5 \textasciitilde\ 2} & \textbf{0.75 \textasciitilde\ 3} & \textbf{1.5 \textasciitilde\ 3} & \textbf{2.25 \textasciitilde\ 3} & \textbf{1 \textasciitilde\ 4} & \textbf{2 \textasciitilde\ 4} & \textbf{3 \textasciitilde\ 4} \\
1.54e-02	&1.33	&1.94	&2.03	&0.92	&1.77	&1.98	&1.17	&1.68	&1.98\\
7.68e-03	&0.58	&1.84	&1.94	&1.17	&1.67	&2.00	&1.71	&1.61	&1.94\\
3.84e-03	&0.18	&1.80	&1.93	&2.04	&1.61	&2.00	&1.33	&1.59	&1.93\\
1.92e-03	&1.58	&1.74	&1.94	&1.46	&1.61	&1.99	&1.34	&1.61	&1.93\\
9.60e-04	&2.82	&1.72	&1.94	&1.40	&1.66	&2.01	&1.44	&1.69	&1.98\\
\noalign{\smallskip}\hline
\end{tabular}
\end{center}
\end{table}

Indeed, in Figure \ref{tab:num.res:flux} and Table \ref{table:num.res:flux}, we can see for $k\in\{1,2\}$ in the first row of results (corresponding to the case $\delta =1$) a constant convergence rate of $k+1$, which is in agreement with Remark \ref{rem:ocv.global} since $\tilde\eta_h \le h^{k+1}|u|_{W^{k+2,\infty}(\T_h)} \Leftrightarrow \delta \ge 1$. From row to row, we can observe a lower pre-asymptotic convergence rate (still above $(k+1)(p-1)$, but close to it in certain cases) which becomes worse as $\delta$ decreases; this is expected since $\tilde\eta_h$ is proportional to $1/\delta$. We note that, as $k$ increases, the asymptotic rates are lower than the expected $h^{k+1}$, which could be due to the asymptotic regime not being achieved yet for these high-order schemes (due to the constants involving higher derivatives of $u$ in \eqref{rem:ocv.global}). The saturation of convergence rate, for $k=2,3$, at a lower rate than $(k+1)(p-1)$ when $\delta$ is very small could also be explained by the fact that the $W^{k+1,p'}$ norm of $\stress(\cdot,\GRAD u)$ explodes as $\delta\to 0$.

\subsection{Non-degenerate potential}\label{subsec:num.res:potential}

We consider $\delta = 0$ (the $p$-Laplacian case), and the potential $u$ is given by
\[
  u(x,y) = \sin\left(\pi x\right)\sin\left(\pi y\right)+(\pi+1)(x+y) \quad \forall (x,y) \in \Omega.
  \]
Since $|\GRAD u|\ge 1$ on $\Omega$, $\tilde\eta_h = \mu_h(1)$ and we should observe a constant convergence rate of $k+1$, which is indeed the case (see Figure \ref{tab:num.res:potential} and Table \ref{table:num.res:potential}). 

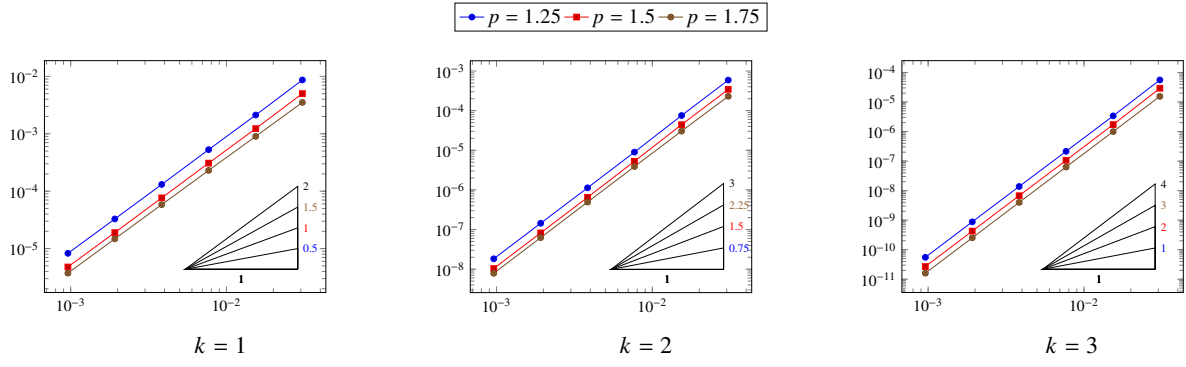
\begin{figure}
  \begin{center}
        \begin{minipage}[b]{0.3\columnwidth}
      \begin{tikzpicture}[scale=0.54]
        \begin{loglogaxis}
          \addplot table[x=meshsize,y=err] {results/mesh1/k1_a1_p125_del0_gradnon0.txt};
          \addplot table[x=meshsize,y=err] {results/mesh1/k1_a1_p150_del0_gradnon0.txt};
          \addplot table[x=meshsize,y=err] {results/mesh1/k1_a1_p175_del0_gradnon0.txt};
          \logLogSlopeTriangle{0.90}{0.4}{0.1}{2}{black};
          \logLogSlopeTriangle{0.90}{0.4}{0.1}{0.5}{blue};
          \logLogSlopeTriangle{0.90}{0.4}{0.1}{1}{red};
          \logLogSlopeTriangle{0.90}{0.4}{0.1}{1.5}{darkbrown};  
        \end{loglogaxis}
      \end{tikzpicture}
      \subcaption*{\hspace{8mm} $k=1$}
    \end{minipage}
    \hfill
        \begin{minipage}[b]{0.3\columnwidth}
      \begin{tikzpicture}[scale=0.54]
        \begin{loglogaxis}[
            legend style = { 
              at={(0.5,1.25)},
              anchor = north,
              legend columns=-1,
              font={\Large}
            }
          ]
          \addplot table[x=meshsize,y=err] {results/mesh1/k2_a1_p125_del0_gradnon0.txt};
          \addplot table[x=meshsize,y=err] {results/mesh1/k2_a1_p150_del0_gradnon0.txt};
          \addplot table[x=meshsize,y=err] {results/mesh1/k2_a1_p175_del0_gradnon0.txt};
          \logLogSlopeTriangle{0.90}{0.4}{0.1}{3}{black};
          \logLogSlopeTriangle{0.90}{0.4}{0.1}{0.75}{blue};
          \logLogSlopeTriangle{0.90}{0.4}{0.1}{1.5}{red};
          \logLogSlopeTriangle{0.90}{0.4}{0.1}{2.25}{darkbrown};
          \legend{$\sob=1.25$ , $\sob=1.5$ , $\sob=1.75$};
        \end{loglogaxis}
      \end{tikzpicture}
      \subcaption*{\hspace{8mm} $k=2$}
    \end{minipage}
    \hfill
        \begin{minipage}[b]{0.3\columnwidth}
      \begin{tikzpicture}[scale=0.54]
        \begin{loglogaxis}
          \addplot table[x=meshsize,y=err] {results/mesh1/k3_a1_p125_del0_gradnon0.txt};
          \addplot table[x=meshsize,y=err] {results/mesh1/k3_a1_p150_del0_gradnon0.txt};
          \addplot table[x=meshsize,y=err] {results/mesh1/k3_a1_p175_del0_gradnon0.txt};
          \logLogSlopeTriangle{0.90}{0.4}{0.1}{4}{black};
          \logLogSlopeTriangle{0.90}{0.4}{0.1}{1}{blue};
          \logLogSlopeTriangle{0.90}{0.4}{0.1}{2}{red};
          \logLogSlopeTriangle{0.90}{0.4}{0.1}{3}{darkbrown};
        \end{loglogaxis}
      \end{tikzpicture}
      \subcaption*{\hspace{8mm} $k=3$}
    \end{minipage}
  \end{center}
   \vspace{-5mm}
   \caption{Numerical results for the test case of Subsection \ref{subsec:num.res:potential}. The steeper slope (in black) indicates the $k+1$ convergence rate. The other slopes indicate the $(k+1)(p-1)$ convergence rate according to $p$.
   \label{tab:num.res:potential}}
\end{figure}

\begin{table}
\caption{Convergence rates for the test case of Subsection \ref{subsec:num.res:potential}. The bold numbers in each column correspond to the $(k+1)(p-1)$ \textasciitilde\ $(k+1)$ convergence rates.}
\label{table:num.res:potential}       % Give a unique label
\begin{center}
\begin{tabular}{c|ccc|ccc|ccccc}
\hline\noalign{\smallskip}
 $k$ & & 1 & & & 2 & & & 3 \\
\hline\noalign{\smallskip}
\diagbox[width=4em]{$h$}{$p$} & 1.25 & 1.5 & 1.75 & 1.25 & 1.5 & 1.75 & 1.25 & 1.5 & 1.75\\
\noalign{\smallskip}\hline\noalign{\smallskip}
3.07e-02 & \textbf{0.5 \textasciitilde\ 2} & \textbf{1 \textasciitilde\ 2} & \textbf{1.5 \textasciitilde\ 2} & \textbf{0.75 \textasciitilde\ 3} & \textbf{1.5 \textasciitilde\ 3} & \textbf{2.25 \textasciitilde\ 3} & \textbf{1 \textasciitilde\ 4} & \textbf{2 \textasciitilde\ 4} & \textbf{3 \textasciitilde\ 4} \\
1.54e-02 & 2.04 & 2.04 & 1.98 & 2.98 & 2.97 & 2.93 & 4.06 & 4.10 & 3.98 \\
7.68e-03 & 2.00 & 1.99 & 1.96 & 3.05 & 3.06 & 2.96 & 3.97 & 4.01 & 3.96 \\
3.84e-03 & 2.01 & 2.01 & 1.98 & 2.99 & 3.02 & 2.98 & 3.95 & 3.98 & 3.98 \\
1.92e-03 & 2.00 & 2.01 & 1.98 & 2.97 & 2.98 & 2.98 & 3.98 & 3.98 & 3.98 \\
9.60e-04 & 1.99 & 1.99 & 1.99 & 2.98 & 2.98 & 2.99 & 3.98 & 3.99 & 3.98 \\
\noalign{\smallskip}\hline
\end{tabular}
\end{center}
\end{table}

\subsection{Non-degenerate flux-potential couple} \label{subsec:num.res:flux-potential}

The exact solution $u$ is given by
\[
  u(x,y) = \sin\left(\pi x\right)\sin\left(\pi y\right) \quad \forall (x,y) \in \Omega.
  \]
Since $\GRAD u$ vanishes at points $(x_i,y_i)_{1 \le i \le 5} \coloneq \{(0,0),(1,0),(0,1),(1,1),(0.5,0.5)\}$, we consider a degeneracy parameter function $\delta$ as the sum of bump functions centered at these points with 0.2 radius. Specifically,
\begin{equation}\label{eq:delta-gauss}
\delta(x,y) = \sum_{i=1}^5\left\{\begin{array}{cl}\exp\left(1-\dfrac{1}{1-25((x-x_i)^2+(y-y_i)^2)}\right) &\text{if } \sqrt{(x-x_i)^2+(y-y_i)^2} < 0.2, \\[.8em] 0 &\text{otherwise.}\end{array}\right.
\end{equation}
As a consequence, $\delta$ vanishes on about three quarters of $\Omega$, however, $\tilde\eta_h = \mu_h(1)$ and we should observe a constant convergence rate of $k+1$. This is confirmed by the results presented in Figure \ref{tab:num.res:flux-potential} and Table \ref{table:num.res:flux-potential}.

\begin{figure}
  \begin{center}
%        \begin{minipage}[b]{0.05\columnwidth}
%      \begin{tikzpicture}
%      	\draw node[scale=0.8,rotate=90]{\hspace*{1.8cm}$\delta$ defined by \eqref{eq:delta-gauss}};
%      \end{tikzpicture}
%    \end{minipage}
%        \hfill
        \begin{minipage}[b]{0.3\columnwidth}
      \begin{tikzpicture}[scale=0.54]
        \begin{loglogaxis}
          \addplot table[x=meshsize,y=err] {results/mesh1/k1_a1_p125_delbump.txt};
          \addplot table[x=meshsize,y=err] {results/mesh1/k1_a1_p150_delbump.txt};
          \addplot table[x=meshsize,y=err] {results/mesh1/k1_a1_p175_delbump.txt};
          \logLogSlopeTriangle{0.90}{0.4}{0.1}{2}{black};
          \logLogSlopeTriangle{0.90}{0.4}{0.1}{0.5}{blue};
          \logLogSlopeTriangle{0.90}{0.4}{0.1}{1}{red};
          \logLogSlopeTriangle{0.90}{0.4}{0.1}{1.5}{darkbrown};  
%          \logLogSlopeTriangle{0.90}{0.4}{0.1}{0.75}{blue};
%          \logLogSlopeTriangle{0.90}{0.4}{0.1}{1}{red};
%          \logLogSlopeTriangle{0.90}{0.4}{0.1}{2.25}{darkbrown};
        \end{loglogaxis}
      \end{tikzpicture}
      \subcaption*{\hspace{8mm} $k=1$}
    \end{minipage}
    \hfill
        \begin{minipage}[b]{0.3\columnwidth}
      \begin{tikzpicture}[scale=0.54]
        \begin{loglogaxis}[
            legend style = { 
              at={(0.5,1.25)},
              anchor = north,
              legend columns=-1,
              font={\Large}
            }
          ]
          \addplot table[x=meshsize,y=err] {results/mesh1/k2_a1_p125_delbump.txt};
          \addplot table[x=meshsize,y=err] {results/mesh1/k2_a1_p150_delbump.txt};
          \addplot table[x=meshsize,y=err] {results/mesh1/k2_a1_p175_delbump.txt};
          \logLogSlopeTriangle{0.90}{0.4}{0.1}{3}{black};
          \logLogSlopeTriangle{0.90}{0.4}{0.1}{0.75}{blue};
          \logLogSlopeTriangle{0.90}{0.4}{0.1}{1.5}{red};
          \logLogSlopeTriangle{0.90}{0.4}{0.1}{2.25}{darkbrown};
          \legend{$\sob=1.25$ , $\sob=1.5$ , $\sob=1.75$};
        \end{loglogaxis}
      \end{tikzpicture}
      \subcaption*{\hspace{8mm} $k=2$}
    \end{minipage}
    \hfill
        \begin{minipage}[b]{0.3\columnwidth}
      \begin{tikzpicture}[scale=0.54]
        \begin{loglogaxis}
          \addplot table[x=meshsize,y=err] {results/mesh1/k3_a1_p125_delbump.txt};
          \addplot table[x=meshsize,y=err] {results/mesh1/k3_a1_p150_delbump.txt};
          \addplot table[x=meshsize,y=err] {results/mesh1/k3_a1_p175_delbump.txt};
          \logLogSlopeTriangle{0.90}{0.4}{0.1}{4}{black};
          \logLogSlopeTriangle{0.90}{0.4}{0.1}{1}{blue};
          \logLogSlopeTriangle{0.90}{0.4}{0.1}{2}{red};
          \logLogSlopeTriangle{0.90}{0.4}{0.1}{3}{darkbrown};
        \end{loglogaxis}
      \end{tikzpicture}
      \subcaption*{\hspace{8mm} $k=3$}
    \end{minipage}
  \end{center}
  \vspace{-5mm}
   \caption{Numerical results for the test case of Subsection \ref{subsec:num.res:flux-potential}. The steeper slope (in black) indicates the $k+1$ convergence rate. The other slopes indicate the $(k+1)(p-1)$ convergence rate according to $p$.
   \label{tab:num.res:flux-potential}}
\end{figure}
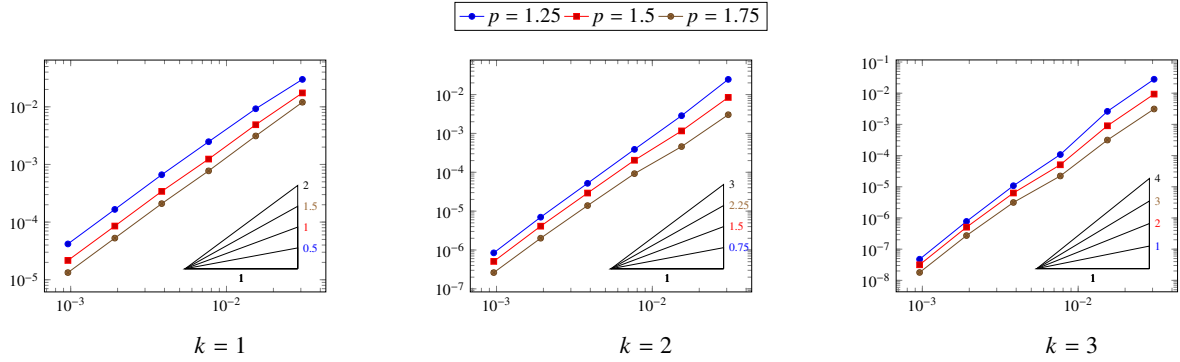

\begin{table}
\caption{Convergence rates for the test case of Subsection \ref{subsec:num.res:flux-potential}. The bold numbers in each column correspond to the $(k+1)(p-1)$ \textasciitilde\ $(k+1)$ convergence rates.}
\label{table:num.res:flux-potential}       % Give a unique label
\begin{center}
\begin{tabular}{c|ccc|ccc|ccccc}
\hline\noalign{\smallskip}
 $k$ & & 1 & & & 2 & & & 3 \\
\hline\noalign{\smallskip}
\diagbox[width=4em]{$h$}{$p$} & 1.25 & 1.5 & 1.75 & 1.25 & 1.5 & 1.75 & 1.25 & 1.5 & 1.75\\
\noalign{\smallskip}\hline\noalign{\smallskip}
3.07e-02 & \textbf{0.5 \textasciitilde\ 2} & \textbf{1 \textasciitilde\ 2} & \textbf{1.5 \textasciitilde\ 2} & \textbf{0.75 \textasciitilde\ 3} & \textbf{1.5 \textasciitilde\ 3} & \textbf{2.25 \textasciitilde\ 3} & \textbf{1 \textasciitilde\ 4} & \textbf{2 \textasciitilde\ 4} & \textbf{3 \textasciitilde\ 4} \\
1.54e-02 & 1.70 & 1.84 & 1.95 & 3.12 & 2.87 & 2.74 & 3.43 & 3.39 & 3.32 \\
7.68e-03 & 1.89 & 1.97 & 2.00 & 2.87 & 2.50 & 2.30 & 4.58 & 4.14 & 3.81 \\
3.84e-03 & 1.90 & 1.86 & 1.89 & 2.90 & 2.80 & 2.73 & 3.32 & 3.01 & 2.82 \\
1.92e-03 & 2.00 & 2.00 & 1.99 & 2.90 & 2.83 & 2.79 & 3.80 & 3.64 & 3.52 \\
9.60e-04 & 1.99 & 1.98 & 1.99 & 3.05 & 3.01 & 2.94 & 4.03 & 4.00 & 3.93 \\
\noalign{\smallskip}\hline
\end{tabular}
\end{center}
\end{table}

\subsection{Degenerate problem}\label{subsec:num.res:degenerate}

We consider $\delta = 0$ (the $p$-Laplacian case), and the potential $u$ is given such that,
\[
u(x,y) = \frac{1}{10}\exp\left(-10\left(|x-0.5|^{p+\frac{k+2}{4}}+|y-0.5|^{p+\frac{k+2}{4}}\right)\right) \quad \forall (x,y) \in \Omega.
\]
The particular choice of $u$, which changes with $p$ and $k$, is driven by the need to ensure that the function and its flux have the required regularity for the error estimate in Theorem \ref{thm:error.estimate} to be valid (and thus to potentially avoid some of the issues observed in Section \ref{subsec:num.res:flux}), all the while not being too smooth or with a simple structure, which might artificially generate a better convergence of the scheme.
Since $\GRAD u$ vanishes on an entire region of the domain and $\delta=0$, we have, $\tilde\eta_h = +\infty$ and we should observe $(k+1)(p-1)$ asymptotic convergence rates.
The results for this test are presented in Figure \ref{tab:num.res:degenerate} and Table \ref{table:num.res:degenerate}. In most cases, they do confirm that the asymptotic rate of convergence is closer to $(k+1)(p-1)$ than $k+1$, with major exceptions for $(p,k)=(1.25,1)$ and $(p,k)=(1.75,3)$, for which the observed rate is closer to $k+1$ -- probably because, for the considered mesh sizes and parameter values, the first term in \eqref{eq:error.estimate.original} might be dominant due to a larger multiplicative constant. Another explanation could be that the asymptotic regime is not reached yet for these tests, or, in view of the recent results in \cite{carstensen:20}, that the error estimate in Theorem \ref{thm:error.estimate}, which is valid in a more general setting, is actually sub-optimal for these particular test cases.
  In any case, these tests are outliers in the results presented here, which otherwise support rather well the theoretical error estimate.

\begin{figure}
  \begin{center}
        \begin{minipage}[b]{0.3\columnwidth}
      \begin{tikzpicture}[scale=0.54]
        \begin{loglogaxis}
          \addplot table[x=meshsize,y=err] {results/mesh1/k1_a1_p125_del0_grad0.txt};
          \addplot table[x=meshsize,y=err] {results/mesh1/k1_a1_p150_del0_grad0.txt};
          \addplot table[x=meshsize,y=err] {results/mesh1/k1_a1_p175_del0_grad0.txt};
          \logLogSlopeTriangle{0.90}{0.4}{0.1}{2}{black};
          \logLogSlopeTriangle{0.90}{0.4}{0.1}{0.5}{blue};
          \logLogSlopeTriangle{0.90}{0.4}{0.1}{1}{red};
          \logLogSlopeTriangle{0.90}{0.4}{0.1}{1.5}{darkbrown};  
        \end{loglogaxis}
      \end{tikzpicture}
      \subcaption*{\hspace{8mm} $k=1$}
    \end{minipage}
    \hfill
        \begin{minipage}[b]{0.3\columnwidth}
      \begin{tikzpicture}[scale=0.54]
        \begin{loglogaxis}[
            legend style = { 
              at={(0.5,1.25)},
              anchor = north,
              legend columns=-1,
              font={\Large}
            }
          ]
          \addplot table[x=meshsize,y=err] {results/mesh1/k2_a1_p125_del0_grad0.txt};
          \addplot table[x=meshsize,y=err] {results/mesh1/k2_a1_p150_del0_grad0.txt};
          \addplot table[x=meshsize,y=err] {results/mesh1/k2_a1_p175_del0_grad0.txt};
          \logLogSlopeTriangle{0.90}{0.4}{0.1}{3}{black};
          \logLogSlopeTriangle{0.90}{0.4}{0.1}{0.75}{blue};
          \logLogSlopeTriangle{0.90}{0.4}{0.1}{1.5}{red};
          \logLogSlopeTriangle{0.90}{0.4}{0.1}{2.25}{darkbrown};
          \legend{$\sob=1.25$ , $\sob=1.5$ , $\sob=1.75$};
        \end{loglogaxis}
      \end{tikzpicture}
      \subcaption*{\hspace{8mm} $k=2$}
    \end{minipage}
    \hfill
        \begin{minipage}[b]{0.3\columnwidth}
      \begin{tikzpicture}[scale=0.54]
        \begin{loglogaxis}
          \addplot table[x=meshsize,y=err] {results/mesh1/k3_a1_p125_del0_grad0.txt};
          \addplot table[x=meshsize,y=err] {results/mesh1/k3_a1_p150_del0_grad0.txt};
          \addplot table[x=meshsize,y=err] {results/mesh1/k3_a1_p175_del0_grad0.txt};
          \logLogSlopeTriangle{0.90}{0.4}{0.1}{4}{black};
          \logLogSlopeTriangle{0.90}{0.4}{0.1}{1}{blue};
          \logLogSlopeTriangle{0.90}{0.4}{0.1}{2}{red};
          \logLogSlopeTriangle{0.90}{0.4}{0.1}{3}{darkbrown};
        \end{loglogaxis}
      \end{tikzpicture}
      \subcaption*{\hspace{8mm} $k=3$}
    \end{minipage}
  \end{center}
  \vspace{-5mm}
   \caption{Numerical results for the test case of Subsection \ref{subsec:num.res:degenerate}. The steeper slope (in black) indicates the $k+1$ convergence rate. The other slopes indicate the $(k+1)(p-1)$ convergence rate according to $p$.
   \label{tab:num.res:degenerate}}
\end{figure}
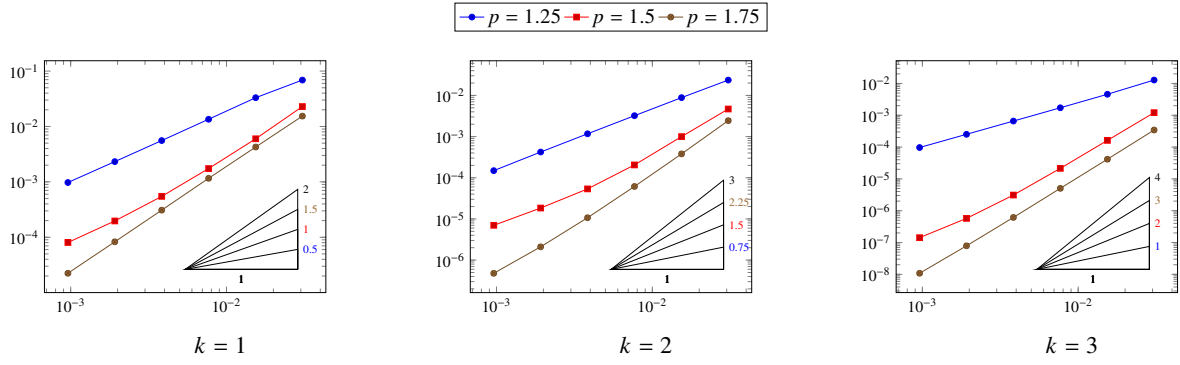

\begin{table}
\caption{Convergence rates for the test case of Subsection \ref{subsec:num.res:degenerate}. The bold numbers in each column correspond to the $(k+1)(p-1)$ \textasciitilde\ $(k+1)$ convergence rates.}
\label{table:num.res:degenerate}       % Give a unique label
\begin{center}
\begin{tabular}{c|ccc|ccc|ccccc}
\hline\noalign{\smallskip}
 $k$ & & 1 & & & 2 & & & 3 \\
\hline\noalign{\smallskip}
\diagbox[width=4em]{$h$}{$p$} & 1.25 & 1.5 & 1.75 & 1.25 & 1.5 & 1.75 & 1.25 & 1.5 & 1.75\\
\noalign{\smallskip}\hline\noalign{\smallskip}
3.07e-02 & \textbf{0.5 \textasciitilde\ 2} & \textbf{1 \textasciitilde\ 2} & \textbf{1.5 \textasciitilde\ 2} & \textbf{0.75 \textasciitilde\ 3} & \textbf{1.5 \textasciitilde\ 3} & \textbf{2.25 \textasciitilde\ 3} & \textbf{1 \textasciitilde\ 4} & \textbf{2 \textasciitilde\ 4} & \textbf{3 \textasciitilde\ 4} \\
1.54e-02 & 1.06 & 1.94 & 1.86 & 1.42 & 2.24 & 2.70 & 1.49 & 2.91 & 3.98 \\
7.68e-03 & 1.29 & 1.78 & 1.87 & 1.45 & 2.28 & 2.62 & 1.40 & 2.91 & 3.96 \\
3.84e-03 & 1.28 & 1.67 & 1.91 & 1.46 & 1.93 & 2.52 & 1.39 & 2.79 & 3.98 \\
1.92e-03 & 1.26 & 1.47 & 1.90 & 1.47 & 1.55 & 2.35 & 1.38 & 2.44 & 3.98 \\
9.60e-04 & 1.25 & 1.29 & 1.89 & 1.50 & 1.39 & 2.14 & 1.38 & 2.01 & 3.98 \\
\noalign{\smallskip}\hline
\end{tabular}
\end{center}
\end{table}

\section{Conclusion}

We have presented and analysed a Hybrid High-Order scheme of arbitrary order $k$, for a non-linear model that generalises the $p$-Laplace equation (with $p\in (1,2]$) through the addition of an offset in the flux, that potentially remove its singularity at $0$. Our error estimate highlights various convergence regimes for the scheme, depending on its degeneracy or lack thereof (the latter occurring in presence of a non-zero offset, or when the gradient of the continuous solution does not vanish); for a degenerate model we recover the known rates of convergence in $(k+1)(p-1)$, while an optimal rate of $(k+1)$, identical to the rate for linear models, is obtained when the model is not degenerate. These regimes are locally driven by a dimensionless number, and intermediate regimes are also identified.

Several numerical tests have been provided, and show a good agreement with the theoretical error estimate, except in a few cases where the convergence appears to be faster than expected (which could be due to the asymptotic regime not yet being attained in that case, or to the specifics of the particular degenerate test case considered here).

\raggedright
\printbibliography

\end{document}